\documentclass[twoside,12pt, leqno]{article}
\usepackage{etex}
\usepackage{amsmath,amscd,amsfonts, amsthm,amssymb,amsxtra,latexsym,epsfig,epic,graphics}
\usepackage[matrix,arrow,curve]{xy}
\usepackage{graphicx}

\usepackage[bookmarksnumbered]{hyperref}
\hypersetup{colorlinks=true,backref=true,citecolor=blue}
\usepackage{array,ragged2e}

\usepackage[top=28mm,bottom=28mm]{geometry}
\usepackage[english]{babel}
\usepackage[utf8]{inputenc}
\usepackage{times}
\usepackage{setspace}
\usepackage{color}
\usepackage{paralist}
\usepackage{pgfplots}
\usepackage{mathrsfs}
\usepackage{mathtools}
\usepackage{bm}
\usepackage{tikz-cd}
\usepackage{tikz}
\usetikzlibrary{chains,
	matrix,
	positioning,
	scopes,arrows
}
\usepackage{pdfpages}
\usepackage{bibentry}
\usepackage{multirow}

\newtheorem{theorem}{Theorem}[section]
\newtheorem{lemma}[theorem]{Lemma}
\newtheorem{proposition}[theorem]{Proposition}
\newtheorem{corollary}[theorem]{Corollary}
\theoremstyle{definition}

\newtheorem{example}[theorem]{Example}
\newtheorem{notation}[theorem]{Notation}
\theoremstyle{remark}
\newtheorem{remark}[theorem]{Remark}

\newcommand{\sO}{{\cal O}}

\renewcommand{\P}{{\mathbb P}}

\DeclareMathOperator{\im}{im}

\DeclareMathOperator{\coker}{coker}

\DeclareMathOperator{\ann}{ann}
\DeclareMathOperator{\Proj}{Proj}

\DeclareMathOperator{\rank}{rank}

\def\ZZ{{\mathbb Z}}
\def\QQ{{\mathbb Q}}
\def\CC{{\mathbb C}}

\def\PP{{\mathbb P}}
\def\AA{{\mathbb A}}

\def\PP{{\mathbb P}}
\def\PP{{\mathbb P}}
\def\FF{{\mathbb F}}

\def\sU{{\cal U}}
\def\sT{{\cal T}}
\def\sL{{\cal L}}

\def\sC{{\cal C}}

\DeclareMathOperator{\betti}{{\rm betti}}

\def\Mac2{{\emph{Macaulay2}}}

\def\canmod{X_{can}}
\DeclareMathOperator{\Tors}{Tors}

\makeatletter
\def\Ddots{\mathinner{\mkern1mu\raise\p@
\vbox{\kern7\p@\hbox{.}}\mkern2mu
\raise4\p@\hbox{.}\mkern2mu\raise7\p@\hbox{.}\mkern1mu}}
\makeatother

\DeclareMathOperator{\Pic}{Pic}

\DeclareMathOperator{\Sym}{Sym}

\makeatletter
\def\Ddots{\mathinner{\mkern1mu\raise\p@
\vbox{\kern7\p@\hbox{.}}\mkern2mu
\raise4\p@\hbox{.}\mkern2mu\raise7\p@\hbox{.}\mkern1mu}}
\makeatother

\usepackage{times}
\newdimen\x \x=12pt
 
\usepackage{color}

\date{}
\title{Marked Godeaux surfaces with special bicanonical fibers}
\author{Frank-Olaf Schreyer and Isabel Stenger}

\begin{document}

\maketitle

\begin{abstract}

In this paper we study marked numerical Godeaux surfaces with special bicanonical fibers. Based on the construction method of marked Godeaux surfaces in \cite{SS20} we give a complete characterization for the existence of hyperelliptic bicanonical fibers and torsion fibers. Moreover, we describe how the families of Reid and Miyaoka with torsion $\ZZ/3\ZZ$ and $\ZZ/5\ZZ$  arise in our homological setting.
\end{abstract}
\tableofcontents
\section*{Introduction}
In \cite{SS20} we presented a new construction method for numerical Godeaux surfaces based on homological algebra. Thereby, we first restricted to the case where
the bicanonical system $|2K_X|$  of a Godeaux surface $X$ has no fixed part and four distinct base points. Fixing an enumeration of the base points $p_0,\ldots,p_3$, we introduced the notion of a \textit{marked numerical Godeaux surface}. Note that any such surface has a torsion group of odd order by  \cite[Lemma 1.9]{SS20}. Thus, by the results of Miyaoka this paper  restricts to surfaces with a torsion group $T = \ZZ/5\ZZ$, $T = \ZZ/3\ZZ$ or trivial $T$. In \cite{SS20} we 
showed that constructing the canonical ring $R(X)$ of a marked numerical Godeaux surface $X$ is basically equivalent to performing two steps:
\begin{description}
	\item[Construction step 1] Choose a line $\ell$ in a complete intersection $Q = V(q_0,\ldots,q_3) \subset \PP^{11}$, where each $q_i$ is the Pfaffian of a $4\times 4$ skew-symmetric matrix $M_i$. 
	\item[Construction step 2] Solve a system of linear equations (depending on $\ell$) using syzygies.
\end{description}
We have shown that for a general line $\ell \subset Q$, we obtain a four dimensional linear solution space in the second step \cite[Theorem 5.1]{SS20}, hence we obtain a $\PP^3$- bundle over an open subset of the Fano variety of lines $F_1(Q)$ which is 8-dimensional. In our setting, a general line in $Q$ is a line which does not intersect special subvarieties of $Q$ of codimension $\geq 2$ which will be introduced later.  Together with a 3-dimensional group operation of the torus $({\CC^*})^3$ on $F_1(Q)$ we obtain in total an 8-dimensional family. Showing that the simply connected Barlow surface (see \cite{Barlow}) is deformation equivalent to an element of this family, gives then the existence of an 8-dimensional family of marked numerical Godeaux surfaces which are simply connected. We call this family the \textit{dominant family} as its corresponding lines in $Q$ dominate $F_1(Q)$.

Blowing up the four base points of $|2K_X|$ we obtain a fibration $\tilde{X} \rightarrow \P^1$, where $\tilde{X}$ denotes the blow-up. We call the fibers as well as their strict transforms the \textit{bicanonical fibers}. 
In this paper, we study marked numerical Godeaux surfaces with special bicanonical fibers.

For a Godeaux surface with torsion group $T = \ZZ/5\ZZ$ or $T = \ZZ/3\ZZ$ there exist special reducible bicanonical fibers. Indeed, if $\tau_i \in \Tors X \subset \Pic(X)$ is non-trivial, then there exists a unique effective divisor $D_i \in |K_X + \tau_i|$ and $D_i + D_{n-i}$ is a reducible fiber in $|2K_X|$, where $n$ denotes the order of the torsion group and $\tau_{n-i} = - \tau_i$. We call such a fiber a \textit{torsion fiber} and give a complete characterization for the existence of these fibers in terms of our homological setting in Theorem \ref{thm_torsionfib}. Moreover, we show that there is a 3-dimensional subset $V_{tors} \subset Q$ which is the union of three $\PP^3$s which lines in $Q$ have to intersect to obtain torsion fibers. For $T = \ZZ/5\ZZ$ and $T = \ZZ/3\ZZ$, as the dimension of the variety of suitable lines in $Q$ drops, the dimension of the solution space in the second step increases  and we reconstruct the 8-dimensional families of these surfaces known by Miyaoka and Reid. 

We show that there is a 5-dimensional variety $V_{hyp} \subset Q \subset \PP^{11}$ such that a bicanonical fiber over a point $q \in \ell \cong \PP^1 \cong |2K_X|$ is hyperelliptic if and only if $q \in V_{hyp}$ (see Theorem \ref{thm_hyplocusminors}). Furthermore, in Theorem \ref{hypLocus} we prove that $V_{hyp}$  is unirational. We prove that there is a 7-dimensional family of Godeaux surfaces with precisely one hyperelliptic fiber, and a 6-dimensional family with two hyperelliptic fibers. Note that the only other known examples of Godeaux surfaces with hyperelliptic bicanonical fibers were the Barlow surfaces (see \cite{Barlow}) and the surfaces introduced by Craighero and Gattazzo (see \cite{CraigheroGattazzo}).

 To obtain a complete classification of all marked Godeaux surfaces, we have to determine the loci of all lines in $Q$ at which the linear solution space in the second Construction step is so big that we obtain another component. 
For the time being, we failed to give a complete characterization of all numerical Godeaux surfaces due to the existence of ghost components. These ghost components arise from other components of the space of all solutions of our deformation problem. We were not able to determine the general stratification of the Fano scheme of lines $F_1(Q)$  with regard to the dimension of the solution space in the second Construction step.\medskip

{\bf Acknowledgments.}  This work was supported by the Deutsche Forschungsgemeinschaft (DFG, German Research Foundation) - Project-ID 286237555 - TRR 195. We thank Wolfram Decker and Miles Reid for inspiring conversations. Our work makes essential use of $\Mac2$ (\cite{M2}).

\section{Preliminaries}\label{sec_prelim}
Throughout this paper, we use the following notation. 
\begin{itemize}
	\item $X$ denotes a numerical Godeaux surface;
	\item $\pi\colon X \rightarrow \canmod = \Proj(R(X))$ denotes the morphism to the canonical model;
	\item $K_X$ and $K_{\canmod}$ denote canonical divisors;
	\item $\Tors X$  denotes the torsion subgroup of the Picard group of $X$;
	\item $\Bbbk$ denotes the ground field.
\end{itemize}
We are mainly interested in the case $\Bbbk = \CC$. For computations we also use $\Bbbk = \QQ$ or number fields and for our experiments with Macaulay2, $\Bbbk$ can also be a finite field which we often may regard as a specialization of a number field.
We will now briefly summarize the main results of \cite{SS20} which will be needed in the following.

Let $x_0,x_1$ be a basis of the bicanonical system $|2K_X|$, and let  $y_0,\ldots,y_3$ be a basis of the tricanonical system $|3K_X|$. We consider the graded polynomial ring 
$$S =  \Bbbk[x_0,x_1,y_0,y_1,y_2,y_3]$$ with $\deg(x_i)= 2$ and $\deg(y_j)=3$. The canonical ring $R(X)$ is a finitely generated $S$-module and we showed that
there exists a minimal free resolution of $R(X)$ as an $S$-module of type
\begin{equation}\label{bettinumbers}
0 \leftarrow R(X) \leftarrow F_0 \xleftarrow{d_1} F_1 \xleftarrow{d_2} F_1^\ast(-17) \xleftarrow{d_1^t} F_0^\ast(-17) \leftarrow 0,
\end{equation} 
where $F_0 = S \oplus S(-4)^4 \oplus S(-5)^3$, $F_1 = S(-6)^6 \oplus S(-7)^{12} \oplus S(-8)^8$ and $d_2$ is alternating (see \cite[Theorem 1.5 and Corollary 5.6]{Stenger19}).

In \cite{SS20} Section 3, we give a complete description of the maps of the complex $F_\cdot$ modulo the exact sequence $x_0,x_1$ and obtain the other entries using an unfolding technique. More precisely, we make an Ansatz for $d_1$ and $d_2$ with new variables (which are a priori unknown expressions in the $x_i$ and $y_j$) and evaluate the condition $d_1d_2 = 0$. Then, finding a solution for the system of equations is equivalent to performing the two steps described in the Introduction. 

Under a certain additional condition,  the module $R := \coker d_1$ carries a ring structure. If this is the case and $\Proj R$ has additionally only Du Val singularities, then $R$ is the canonical ring of a (marked) numerical Godeaux surface $X$ (see \cite[Theorem 5.02]{Stenger18}).  Unfortunately, these conditions are only sufficient. Thus, there are solutions of our deformation problem leading to surfaces which might not be Godeaux surfaces but form families of dimension $\geq 8$.  We call these components of the solution space \textit{Ghost components}. Perhaps these are stable Godeaux surfaces in the sense of \cite{FranParRoll}.

 We briefly recall the general notation for the unfolding parameters for $d_1$ and $d_2$ from \cite{SS20}.
\begin{notation}[The general set-up]\label{not_gensetup}
	We write 
	\begingroup
	\renewcommand*{\arraystretch}{1.1}
	\begin{align*}
		d_1& =   \begin{array}{c|c|c|c}
			& 6S(-6) & 12S(-7) & 8S(-8) \\ \hline
			S &  \color{red} b_0(y)  \color{blue} + \ast&  \color{blue} \ast &  \color{blue} \ast \\ \hline
			4S(-4) & \color{blue} a  & \color{red} b_1(y) & \color{blue} c \\ \hline
			3S(-5) & \color{black} 0 & \color{blue} e & \color{red}b_2(y)
		\end{array} \\ \\
		d_2& = \begin{array}{c|c|c|c}
			& 6S(-11) & 12S(-10) & 8S(-9) \\ \hline
			6S(-6) & \color{blue} o & \color{blue} n & \color{red}b_3(y) \\ \hline
			12S(-7) & \color{blue} -n^t  & \color{red}{\color{red} b_4(y)} & \color{blue} p \\ \hline
			8S(-8) & \color{red}-b_3(y)^t & \color{blue} -p^t &
			\color{black} 0
		\end{array} 
	\end{align*}
	\endgroup
with skew-symmetric matrices $o$ and $b_4$. The entries of the red matrices are known \cite[Proposition 3.6]{SS20} and the entries of the blue matrices are variables of the corresponding degree. For example, the $a$-variables have all degree 2 and a possible assignment is hence a linear combination  of $x_0,x_1$. 
\end{notation}
\begin{remark}
Let $d_1'$ be the matrix obtained from $d_1$ by erasing the first row. Note that is enough to solve the system $d_1'd_2 = 0 $ with $d_1'$ and $d_2$ as in Notation \ref{not_gensetup} since we can recover the missing first row of $d_1$ from the skew-symmetric $d_2$ by a syyzy computation. 
\end{remark}
Evaluating the relations coming from $d_1'd_2 = 0 $, we have seen in \cite{SS20} that any $p$- and $e$-variable can be expressed by some  $a$-variable and that there are a priori only 12 non-zero $a$-variables and an $a$-matrix of the form
\begin{center}
	\begingroup
	\color{blue}
	$a =  \left(\begin{array}{cccccc}
	a_{0,1}&
	a_{0,2}&
	0&
	a_{0,3}&
	0&
	0  \\
	a_{1,0}&
	0&
	a_{1,2}&
	0&
	a_{1,3}&
	0 \\
	0&
	a_{2,0}&
	a_{2,1}&
	0 &
	0&
	a_{2,3}\\
	0&
	0&
	0&
	a_{3,0}&
	a_{3,1}&
	a_{3,2} \end{array}\right)$\endgroup. \end{center}
As a possible entry of the matrix $a$ is a linear combination of $x_0,x_1$ with coefficients in $\Bbbk$,  we will think of  these coefficients as Stiefel coordinates, hence as the entries of $2\times 12$-matrices having at least one non-vanishing maximal minor. Moreover, these Stiefel coordinates define a line $\ell$ in the $\PP^{11}$ of $a$-variables. 
Again, from evaluating $d_1'd_2 =0$, we see that the 12 $a$-variables must satisfy exactly four quadratic equations
\begin{align*}
q_0 &=  a_{1,2} a_{1,3}\color{black} -a_{2,1} a_{2,3}+a_{3,1} a_{3,2}, \\[8pt]
q_1 &=  a_{0,2} a_{0,3}
-a_{3,0} a_{3,2}
+a_{2,0} a_{2,3}, \\[8pt]
q_2 &= a_{1,0} a_{1,3} -   a_{0,1} a_{0,3}
+a_{3,0} a_{3,1}, \\[8pt]
q_3 &= a_{0,1} a_{0,2}-a_{1,0} a_{1,2}+a_{2,0} a_{2,1}.
\end{align*}
The $q_{i}$ are  Pfaffians of  skew-symmetric matrices $M_0,\ldots,M_3$ of size 4. For a possible choice of matrices $M_i$ we refer to \cite[Section 3]{SS20}. 
The corresponding variety $Q = V(q_0,\ldots,q_3) \subset \P^{11}$  is an irreducible complete intersection. 

After the choice of a parametrized line $\ell \subset Q$, hence the choice of $a$, $e$ and $p$, 
the second step of our construction consists in solving the remaining equations. In \cite{SS20} we have seen that there are exactly 42 relations which are linear in the unknown $o$-variables and $c$-variables. We represent these equations in a matrix 
 $$
 \left(\begin{array}{c| c}
 	0  & l_1 \\ \hline l_2 & q
 \end{array}\right) 
 \left(\begin{array}{c}
	$\underline{c}$ \\  $\underline{o}$
\end{array}\right) = 0
$$
 where $l_1$ is a $12 \times 12$-matrix and $l_2$ is a $30 \times 20$-matrix, both having entries linear in the $a$-variables, and $q$ is a 
 $30 \times 12$-matrix with quadratic entries. Note that $l_1$ is a direct sum of 4 skew-symmetric matrices of size 3 whose entries correspond to the rows of the $a$-matrix.
 \begin{notation}
 	We denote  the $42\times 32$-matrix by $m_a$ and 
the standardly graded polynomial ring of the 12 remaining $a$-variables by $S_a$. The quotient ring $S_a/I(Q)$ is denoted by $S_Q$. 
 \end{notation}
Resolving $l_1$ and $l_1^t$ (respectively $l_2$ and $l_2^t$) over the quotient ring $S_Q$ we get a complex $C_1$ (respectively $C_2)$. Note that, rather surprisingly, $\ker(l_2 \otimes S_Q)$ and $\ker(l_2^t \otimes S_Q)$ are both free over $S_Q$. 
Moreover, $\ker(m_{a}^{t}\otimes S_{Q}) \cong 4 S_{Q}(-3)\oplus 12 S_{Q}(-4)$ is free, however $E=\ker(m_{a}\otimes S_{Q})$ is not free. Note that $E$ describes the solution space of our deformation problem. Combining these two complexes, we get a commutative diagram with three split exact rows and generically exact columns:
\begin{center}
\begin{tikzpicture}\label{C-complex}
	\matrix(m)[matrix of math nodes,
		row sep=1.9em, column sep=2.9em,
		text height=1.5ex, text depth=0.25ex]
		{ & 0 & 0 & 0& \\
	       0 &2S_Q(-1) & E & 4S_Q& \\
	       0&20S_Q(2) & 20S_Q(2) \oplus 12S_Q(1) & 12S_Q(1) & 0\\
	       0&30S_Q(3) & 12S_Q(2) \oplus 30S_Q(3) & 12S_Q(2) &0\\
	       0&12S_Q(4) & 4S_Q(3) \oplus 12S_Q(4) & 4S_Q(3)&0\\
	       &0 & 0 & 0 &\\};
	\path[->]
	(m-2-1) edge (m-2-2)
	(m-3-1) edge (m-3-2)
	(m-4-1) edge (m-4-2)
	(m-5-1) edge (m-5-2)
	(m-3-4) edge (m-3-5)
	(m-4-4) edge (m-4-5)
	(m-5-4) edge (m-5-5)
	(m-1-2) edge (m-2-2)
	(m-2-2) edge (m-3-2)
	        edge (m-2-3)
	(m-3-2) edge node[right] {$l_2$} (m-4-2)
	        edge (m-3-3)
	(m-4-2) edge[red] node[right] {$g_2$} (m-5-2)
	        edge[red] (m-4-3)
    (m-5-2) edge (m-6-2)
            edge (m-5-3)
    (m-1-3) edge (m-2-3)
    (m-2-3) edge (m-3-3)
            edge (m-2-4)
   	(m-3-3) edge[red]  node[right]{$m_a$} (m-4-3)
   	(m-4-3) edge (m-5-3)
   	        edge (m-4-4)
    (m-5-3) edge (m-6-3)
            edge (m-5-4)
    (m-1-4) edge (m-2-4)
    (m-2-4) edge[red](m-3-4)
    (m-3-4) edge node[right] {$l_1$} (m-4-4)
     (m-3-3)  edge[red]  (m-3-4)
    (m-4-4) edge node[right] {$g_1$}  (m-5-4)
    (m-5-4) edge (m-6-4);
  	\end{tikzpicture} 
\end{center}
We will show computationally that the support of the homology groups $H_{i}(C_{1})$ and $H_{i}(C_{2})$
 have codimension $\ge 2$ in $Q$. Thus a general line $\ell\subset Q$ does not intersect this locus. If $\ell$ is such a line then the proof of  \cite[Theorem 5.1]{SS20} shows that there is an exact sequence of global sections on $\PP^1 \cong \ell$ 
  \begin{equation}
 0 \rightarrow H^0(\sO_{\PP^1}(-1)^2) \rightarrow H^0(E|_{\ell}) \rightarrow H^0(\sO_{\PP^1}^4) \rightarrow 0
 \end{equation}
 and we obtain a 4-dimensional linear solution space in the second step of the construction.

 In the following we report our results on the homology loci of the complexes introduced above using \Mac2. 
The complex $C_1$ has just homology at the zeroth position and we compute that the three  entries of each row of the $a$-matrix form a regular sequence. More precisely, 
$H_0(C_1) = \coker g_1$ is supported at a 4-dimensional scheme  in $Q$ which decomposes in the union of  eight irreducible components. Indeed, setting the entries of one row of the $a$-matrix to zero, the corresponding locus in $Q$ decomposes into a union of two varieties with Betti tables
\begin{equation*} 
\begin{matrix}
	&0&1\\\text{0:}&1&3\\\text{1:}&\text{.}&4\\\text{2:}&\text{.}&1\\\end{matrix}
	 \ \quad \quad \text{ and \quad \quad }
	\begin{matrix}	&0&1\\\text{total:}&1&7\\\text{0:}&1&6\\\text{1:}&\text{.}&1\\
	& & \end{matrix}.
\end{equation*}
For example, restricted to the components corresponding to the entries of the first row, the $a$-matrix is of the form
\begin{equation*}
\begin{pmatrix}
0&0&0&0&0&0\\
a_{1,0}&0&a_{1,2}&0&a_{1,3}&0\\
0&a_{2,0}&a_{2,1}&0&0&a_{2,3}\\
0&0&0&a_{3,0}&{a}_{3,1,3}&a_{3,2}\end{pmatrix} \quad \hbox{or}\quad  
\begin{pmatrix} 
0&0&0&0&0&0\\
0&0&a_{1,2}&0&a_{1,3}&0\\
0&0&a_{2,1}&0&0&a_{2,3}\\
0&0&0&0&a_{3,1}&a_{3,2}\end{pmatrix}.
\end{equation*}

For the second complex $C_2$ we have $H_2(C_2) = H_3(C_2) = 0$ by construction. The module $H_1(C_2) = \ker g_2 / \im l_2$ is supported at a  5-dimensional scheme of degree $72$ in $Q$.   There are exactly 24 irreducible components all of codimension 6 in $\PP^{11}$ with only two different Betti tables
\begin{equation*} 
\begin{matrix}
 &0&1\\\text{0:}&1&5\\\text{1:}&\text{.}&1\\
 \end{matrix}
 \quad \quad \text{ and } \quad 
  \begin{matrix}
 &0&1\\\text{0:}&1&4\\\text{1:}&\text{.}&2\\\end{matrix}
\end{equation*}
and  exactly 12 components for each Betti table. 
We checked that all components with the same Betti table are equivalent under the $S_4$-operation on the coefficients described in \cite[Proposition 4.4]{SS20} and give only one example for each class here:
\begin{align*}
& \left(a_{0,3},a_{1,3},a_{2,3},a_{3,1},a_{3,2},a_{2,1}a_{2,0}-a_{1,2}a_{1,0}+a_{0,2}a_{0,1}
\right),
\\ 
& \left(a_{0,3},a_{1,2},a_{2,3},a_{3,2},a_{2,1}a_{2,0}+a_{0,2}a_{0,1},
a_{3,1}a_{3,0}+a_{1,3}a_{1,0} \right).
\end{align*}

The module $H_0(C_2) = \coker g_2$ is supported at a 5-dimensional scheme of degree $72$ in $Q$. By inspecting this locus, we see that its minimal primes are exactly the minimal primes of the  $4\times 4$-minors of the $a$-matrix in $Q$ and decomposes  into 5 components with Betti numbers 
$$\hbox{ 4 of type }\quad \begin{matrix}
      &0&1\\\text{0:}&1&6\\\text{1:}&\text{.}&1\cr \cr \\\end{matrix} \quad \hbox{and dimension 4, and the fifth of type } \quad
\begin{matrix}
      &0&1\\\text{0:}&1&\text{.}\\\text{1:}&\text{.}&4\\\text{2
      :}&\text{.}&4\\\text{3:}&\text{.}&4\\\end{matrix}.$$
 In particular we deduce from these computations that the homology loci of $C_{1}$ and $C_{2}$ have codimension $\ge 2$ in $Q$. 
 
 To end this section, we will briefly explain a method how to construct lines in $Q$ meeting no or some special loci different from the one explained in \cite[Section 4]{SS20}. The idea is the following: first we choose a point $p \in Q$. Then, for a general $p \in Q$, the variety $Z = Q \cap T_pQ$ is a cone over a surface. Afterwards we choose a (general) point $q$ in $Z$ different from $p$.  Now, as $Q$ is a complete intersection of quadrics, the line $\ell = \overline{pq}$ is completely contained in $Q$. Proceeding like this, the constructed line $\ell$ does not meet any of the codimension $\geq 2$ subloci of $Q$ introduced above. If we want to construct lines meeting one or two loci, we simply choose the two spanning points in these loci (if possible).

\section{Special bicanonical fibers}\label{special bi-canonical fibers}
In this section we characterize special bicanonical divisors of a marked Godeaux surface $X$. More precisely, we give equivalent conditions for the existence of hyperelliptic (respectively torsion) fibers in terms of some submatrices of the first syzygy matrix $d_1$ of $R(X)$.   Recall the following result of Catanese and Pignatelli on bicanonical divisors:
 \begin{lemma}[\cite{CatanesePignatelli}, Lemma 1.10]\label{lem_bicanchar}
 	Let $X$ be a numerical Godeaux surface and let 
  $C \in |2K_{\canmod}|$. Then one of the following holds:
 \begin{enumerate}[(i)]
 \item $C$ is embedded by $\omega_C$ and $\phi_3(C) = \phi_{\omega_C}(C)$ is the complete intersection of a quadric and a cubic.
 \item $C$ is honestly hyperelliptic and $\phi_3(C) = \phi_{\omega_C}(C)$ is a  twisted cubic curve.
 
 \item $\pi^{*}C=D_{1}+D_{2}$ with $D_{i}\in |K_{X} +\tau_{i}|$, $\tau_{i} \in \Tors X$ nontrivial, $\tau_{1}+\tau_{2}=0$.
 \end{enumerate}
  Case $(i)$ is the general one. 
 \end{lemma}
 
 Note that a Gorenstein curve $C$ is called \textit{honestly hyperelliptic} if there exists a finite morphism $C \rightarrow \P^1$ of degree 2. This definition does not require that $C$ is smooth or irreducible.  

 \begin{proof}
 We follow the proof of \cite{CatanesePignatelli}, Lemma 1.10. Suppose $C$ is not $3$-connected. Then $\pi^{*}C \in |2K_{X}|$ is not $3$-connected as well by \cite{CataneseFranciosiHulekReid} Lemma 4.2, and we have a decomposition
 $$\pi^{*}C= D_{1}+D_{2} \hbox{ with } D_{1}D_{2}\le 2 \hbox{ and } K_{X}D_{i}=1.$$
 Hence $D_{1}^{2}+D_{2}^{2} =(2K_{X})^{2}-D_{1}D_{2}\ge 0$, and  we may assume that $D_{1}^{2}$ is non-negative. $D_{1}^{2}$ is odd, since $D_{1}(D_{1}+K_{X})$ is even. Hence $D_{1}^{2}>0$ and the algebraic Hodge index theorem implies $D_{1}^{2}=1$ and
 $D_{1}=K_{X}+\tau_{1}$ with $\tau_{1}\in  \Tors X \setminus \{0\}$, hence $D_{2}=K_{X}-\tau_{1}$.
 So if we are not in case {\it (iii)}, we may assume that $C$ is 3-connected and the result follows from \cite{CataneseFranciosiHulekReid} Theorem 3.6. 
 \end{proof}

We start characterizing \emph{torsion fibers}, i.e., fibers of type {\it (iii)}.

\begin{theorem}\label{thm_torsionfib} Let $X$ be a marked numerical Godeaux surface with standard resolution $\bf{F}$. A point
$q=(q_{0}:q_{1}) \in \PP^{1}\cong |2K_{X}|$ corresponds to a torsion fiber if and only if $\rank(e(q)) \le 2$.
\end{theorem}
\begin{proof}  Let $\varphi\colon \canmod \rightarrow Y \subset \PP(2^2,3^4)$ be our model in the weighted $\PP^{5}$. Suppose $\rank(e(q)) \le 2$. 
Then $\hat q =(q_{0}:q_{1}:0:0:0:0)$ is a point where the matrix $d_{1}$ drops rank, because
$$ d_1(\hat q) = \left(  \begin{array}{c|ccc |ccc}
    \ast & & 0 & & & \ast & \\ \hline
    a(q)    & & 0 & & & c(q) &\\ \hline
    0    & & e(q) & & & 0 & 
  \end{array} \right).
  $$
So $\hat q \in V(\ann \coker d_{1}) \cap V(y_{0}, \ldots, y_{3}) \subset Y$ corresponds to a base point of $|3K_{\canmod}|$.
By Miyaoka's results (see \cite[Section 3]{Miyaoka}) $\hat q$ is the intersection point 
$\hat D_{1}\cap\hat D_{2}$, where $D_{1}+D_{2} \in |2K_{X}|$  is a torsion fiber
which coincides with the fiber over $q$. Thus $\rank(e(q)) \le 2$ is sufficient. 

The proof that the condition is also necessary is more subtle. The ideal $I_{3}(e)+I(Q) \subset S_{a}$ of $3\times3$-minors of $e$ decomposes in 7 components with Betti numbers
$$\begin{matrix}
      &0&1\\\text{total:}&1&7\\\text{0:}&1&6\\\text{1:}&\text{.}&1\\\end{matrix}\quad \hbox{ or } \quad
      \begin{matrix}
      &0&1\\\text{total:}&1&8\\\text{0:}&1&8\cr \\\end{matrix}
$$
and degree $2$ and $1$ respectively.
We  have encountered the 4 components of the first type already in Section \ref{sec_prelim}. The 3 components of the second type lead to torsion surfaces. Below we will show that there are $8$-dimensional families of $\ZZ/3\ZZ$ and $\ZZ/5\ZZ$ marked Godeaux surfaces, where the $e$ matrix drops rank in 1 respectively 2 points of the loci of the second type.
Since the family of torsion $\ZZ/3\ZZ$ and $\ZZ/5\ZZ$ Godeaux surfaces is known to be irreducible by \cite{ReidGodeaux78} and \cite{Miyaoka} and the condition of being marked in \cite[Proposition 3.4]{SS20} is an open condition, we see that intersecting one or two loci of
the second kind is a necessary condition for torsion marked Godeaux surfaces.
\end{proof}

As a corollary of the proof we obtain the following:

\begin{proposition}\label{excludedCompo} Let $\ell \subset \PP^{11}$ be a line which intersects  $V(I_{3}(e))\cap Q$ in a point of a degree $2$ component which is not contained in a degree $1$ component. Then $\ell$ does not lead to a numerical Godeaux surface. \qed
\end{proposition}

\begin{remark} In Section \ref{ghost} we will see that such lines do lead to surfaces which however are always reducible.
\end{remark}
\begin{theorem}\label{thm_hyplocusminors}
Let $X$ be a marked numerical Godeaux surface, and let $q \in \PP^1$. Then $\rank(a(q)) = 3$ if and only if the corresponding fiber $C_q \in |2K_{\canmod}|$ is hyperelliptic.  
\end{theorem}

\begin{proof} After applying a linear change of coordinates if necessary, we may assume that $q=(0:1)$ and
$$C = \Proj(R(X)/(x_0)).$$  
Furthermore, as  $h^1(X,\sO_X(nK_X)) = 0$ for all $n$, we get 
   \[R(X)/(x_0) \cong \bigoplus_{n \geq 0} H^0(C, \sO_C(nK_{X}|{_C})).\]
Using this, we compute  that $h^0(C, \sO_C(6K_{X}|{_C})) = h^0(C,\sO_C(2K_C)) = 9$ and thus, there are 6 relations between the 15 global sections 
\[x_1^3, \{x_1z_j\}_{0 \leq j \leq 3}, \{y_iy_j\}_{0 \leq i \leq j \leq 3} \in H^0(C, \sO_C(6K_{X}|{_C})). \]
These relations are given by the first 6 columns of $\tilde{d_1} := d_1 \otimes S/(x_0)$.

If $C_{q}$ is hyperelliptic, then by Lemma \ref{lem_bicanchar} there are 3 equations among these relations which are quadrics in the $y_{i}$'s alone. Thus $\rank a(q)\le 3$. 
By computation we see that $I_{3}(e) +I(Q) \subset I_{3}(a)+I(Q)$, hence a point with $\rank a(q) \le 2$ belongs to a torsion fiber.

Conversely, suppose that $\rank a(q)=3$.
Then there exists a three-dimensional space of linear combinations of the first 6 entries 
$$
y_0y_1 + \alpha_1 x_1^3, \ldots, y_2y_3 + \alpha_6 x_1^3
$$ 
of the first row of  $\tilde{d_1}$ in the ideal of $C_{q}$ in $R(X)/(x_0)$. We have to show that
these equations are quadratic in the $y_{i}$'s alone and that they define a rational normal curve.

The ideal $I_{4}(a)+I(Q)$ of $4\times 4$ minors of $a$ decomposes into 5 components with Betti numbers
$$
\hbox{ four of type }\quad \begin{matrix}
      &0&1\\\text{total:}&1&7\\\text{0:}&1&6\\\text{1:}&\text{.}&1\cr \cr \\\end{matrix} \quad \hbox{ and the fifth of type } \quad
\begin{matrix}
      &0&1\\\text{total:}&1&12\\\text{0:}&1&\text{.}\\\text{1:}&\text{.}&4\\\text{2
      :}&\text{.}&4\\\text{3:}&\text{.}&4\\\end{matrix}. 
 $$
The four degree $2$ components belong also to $I_{3}(e) + I(Q)$, hence are excluded, since they do not lead to marked Godeaux surfaces by Proposition   \ref{excludedCompo} unless the point $q$ lies also in a linear component of  $I_{3}(e)  + I(Q)$. Since $\rank a(\tilde q)\le 2$ for any $\tilde q$ contained in one of the linear components of $I_{3}(e)+I(Q)$, the fiber over $q$ is not a torsion fiber, because $\rank a(q)=3$. Since there are at least
2 quadrics in the $y_{i}$'s alone and $C_{q}$ is not at a torsion fiber, it must be a hyperelliptic fiber, and there are actually
precisely 3 quadrics in the $y_{i}$'s alone by Lemma \ref{lem_bicanchar}. 
\end{proof}

We denote by $J_{hyp} \subset S_{a}$ the fifth component of $I_{4}(a)+I(Q)$.
From the proof of Theorem \ref{thm_hyplocusminors} we obtain the following:
\begin{corollary}\label{hyplocus} Let $q\in \ell \subset Q \subset \PP^{11}$ be a point on a line in $Q$ and let $\canmod$ be a Godeaux surface constructed from $\ell$.
Then $C_{q} \subset \canmod$ is a hyperelliptic fiber if and only if $q \in V(J_{hyp})$. \qed
\end{corollary}
Next we describe the image of a hyperelliptic fiber under the birational morphism $\varphi\colon \canmod \rightarrow Y$. We already know that a hyperelliptic curve is mapped 2-to-1 under the rational map $\phi_2\times \phi_3\colon X \dashrightarrow \PP^1 \times \PP^3$. 

\begin{proposition}
	Let $C \in |2K_{\canmod}|$ be a honestly hyperelliptic curve. Then the restriction of $\varphi \colon \canmod \rightarrow Y$ to $C$ is a birational morphism onto its image $G$ with $p_a(G) = p_a(C)+1$.
\end{proposition}

\begin{corollary} If there exists a hyperelliptic fiber, then the morphism $\varphi\colon \canmod \to Y \subset \PP(2^{2},3^{4})$
is not an isomorphism.\qed
\end{corollary}

\begin{proof}[Proof of Proposition]
	The idea of the proof is to embed $C$ (respectively $G$) into projective spaces and show that the induced morphism is  the restriction of a projection from a point on a rational normal scroll. 
	
   Let $K_0$ be the Cartier divisor corresponding to the $g^1_2$ on $C$.  
We consider the very ample line bundle $\sO_C(6K_0) = \sO_C(K_0) \otimes \sO_C(5K_0)$ which is the restriction of $\sO_X(6K_X)$ to $C$ and set 
 $V = H^0(C,\sO_C(6K_0))$. 
Hence $C$ embedded in $\P(V) \cong \P^8$ is contained in a smooth rational normal scroll of type $S(6,1)$ (as $H^0(C,\sO_C(5K_0)) = \Sym^5(H^0(C,\sO_C(K_0)) + \langle u \rangle$ for some global section $u$). The divisor class of $C$ in $S(6,1)$
of type $2H - 2L$, where $L$ denotes the class of a ruling and $H$ the class of a hyperplane section.  Indeed, let $C \sim aH+bL$ for some $a,b$. Then, clearly $a =2$ by the definition of the scroll and $C$ being hyperelliptic, whereas  $b=-2$ follows from the fact that $C$ has degree 12 under the morphism induced by $6K_X$. 
In particular, $C$ does not meet the directrix $H-6L$ of the scroll. 

  Now we consider the image $G \subset \P = \PP(2,3^4)$ of $C$ under the birational morphism $\varphi \colon \canmod \rightarrow Y$.  
Denoting by $R_i$ a rational normal curve of degree $i$ in $\P^i$, we have a diagram 
  \begin{equation*}
  \begin{tikzpicture}[baseline=(current  bounding  box.center)]
  \matrix (m) [matrix of math nodes,row sep=1.6em,column sep=3.9em,minimum width=2em,text height=1.5ex, text
  depth=0.25ex]
  {   C & & S(6,1) \subset \P^8 \\
  	 & G  & S(6,0)\subset \P^7 \\
   R_3 & & \quad R_6 \subset \P^6\\};
  \path[-stealth]
  (m-1-1) edge (m-2-2)
                edge node[left] {$K_C$} (m-3-1)
                edge [commutative diagrams/hook] (m-1-3)
  (m-2-2)  edge[commutative diagrams/hook] node[above] {$\sO_{\PP}(6)|_G$} (m-2-3)
                  edge (m-3-1)
   (m-3-1) edge (m-3-3)
  ;
  \path[-  stealth]
  (m-1-3) edge[dashed] (m-2-3)
  (m-2-3) edge[dashed] (m-3-3)
  ;
  \end{tikzpicture}.
\end{equation*}
The image of  $G$ is contained in a cone $S(6,0)$ over $R_ 6 \subset \P^6$. 
The morphism $C \rightarrow G$ is the restriction of the projection from a point $p$ on the directrix of $S(6,1)$ to $S(6,0)$. Moreover, the unique line of the ruling of $S(6,1)$ through $p$ is the only line through $p$ which intersects $C'$ in two points (counted with multiplicity). Let $H'$ denote the class of a hyperplane section on $S(6,0)$.  Since  $G$ does not meet the vertex of the cone $S(6,0)$ and  $\deg(G) = 12$ in $\P^7$, we see that $G' \sim 2H'$. Hence, using adjunction, we obtain 
 $p_a(G)= 5$. 
So $\varphi|_C \colon C \rightarrow G$ is birational, and $p_a(G) = p_a(C)+1$ holds.
\end{proof}

\subsection{The hyperelliptic locus}\label{subsec_hyplocus}
By Corollary \ref{hyplocus} a point $p \in \ell$ corresponds to a hyperelliptic curve in $|2K|$ if and only if $p\in V_{hyp}=V(J_{hyp})$.
 The ideal  $J_{hyp}$  has codimension $6$ and its free resolution has Betti table
 
 \begin{equation*}
\begin{matrix}
       &0&1&2&3&4&5&6\\\text{total:}&1&12&50&120&153&92&20\\\text{0:}
&1&\text{.}&\text{.}&\text{.}&\text{.}&\text{.}&\text{.}\\\text{1:}&\text
       {.}&4&\text{.}&\text{.}&\text{.}&\text{.}&\text{.}\\\text{2:}&\text
       {.}&4&6&\text{.}&\text{.}&\text{.}&\text{.}\\\text{3:}&\text{.}&4&44&
       40&4&\text{.}&\text{.}\\\text{4:}&\text{.}&\text{.}&\text{.}&80&149&
       92&20\\\end{matrix}.
\end{equation*}
The generators of $J_{hyp}$ can be obtained by calling the function 
\href{https://www.math.uni-sb.de/ag/schreyer/images/data/computeralgebra/M2/doc/Macaulay2/NumericalGodeaux/html/_precomputed__Hyperelliptic__Locus}{precomputedHyperellipticLocus}).

\begin{theorem}\label{hypLocus} $V_{hyp}\subset \PP^{11}$ is birational to a product of a Hirzebruch surface $F=\PP(\sO_{\PP^1}\oplus \sO_{\PP^1}(2))$  with 3 copies of $\PP^1$.
In terms of homogeneous coordinates $v_0,\ldots,z_1$ on the product, where $v_0$ corresponds to the section of $H^0(F,\sO_F(1))\cong H^0( \sO_{\PP^1}\oplus \sO_{\PP^1}(2))$ obtained from the first summand and $v_1$ corresponds to the section obtained from the second summand, $w_0,w_1$ denote pull backs of homogeneous coordinates of the base of $F$ and similarly $x_0,\ldots, z_1$
homogeneous coordinates on the three $\PP^1$ factors, we obtain the following: The birational map
 $$\varphi: \PP(\sO_{\PP^1}\oplus \sO_{\PP^1}(2))\times \PP^1 \times \PP^1 \times \PP^1 \dasharrow V_{hyp}\subset \PP^{11}$$ 
is given by the transpose of the $12 \times 1$ matrix
$$
\begin{pmatrix}
a_{3,2}\cr a_{3,1} \cr a_{3,0} \cr a_{2,3} \cr a_{2,1} \cr a_{2,0} \cr
a_{1,3} \cr a_{1,2} \cr a_{1,0} \cr a_{0,3} \cr a_{0,2} \cr a_{0,1} 
\end{pmatrix}=
\begin{pmatrix}
    v_0^2  (x_1w_1-x_0w_0)   x_1  (x_1+x_0)  y_0  y_1  z_1^2   \cr
    v_0^2 (x_1  w_1-x_0  w_0)x_0  (x_1+x_0 ) y_0  y_1  z_0^2 \cr
    -v_1^2  w_0^2  w_1^2  (w_1+w_0)  x_0  x_1  (x_1+x_0)  y_0  y_1  z_1^2   \cr
    -v_0^2 (x_1  w_1-x_0  w_0)  x_0  x_1  y_0^2  z_1^2                           \cr
     -v_0^2 (x_1  w_1-x_0  w_0)   x_0  (x_1+x_0)  y_1^2  z_0^2                      \cr
      v_1^2  w_0^2  w_1  (w_1+w_0)^2  x_0  x_1  (x_1+x_0)  y_1^2  z_1^2           \cr
      -v_0^2 (x_1  w_1-x_0  w_0)  x_0  x_1  y_0^2  z_0  z_1                        \cr
       v_0^2 (x_1  w_1-x_0  w_0)  x_1  (x_1+x_0)  y_1^2  z_0  z_1                      \cr
      -v_1^2  w_0  w_1^2  (w_1+w_0)^2  x_0  x_1  (x_1+x_0)  y_1^2  z_0  z_1    \cr
      -v_0 v_1  w_0  w_1 (x_1w_1-x_0w_0) x_0  x_1  y_0^2  z_1^2            \cr
      v_0 v_1  w_0  (w_1+w_0)  (x_1  w_1-x_0  w_0)x_1  (x_1+x_0)  y_1^2  z_1^2        \cr
      -v_0 v_1  w_1  (w_1+w_0)  (x_1w_1-x_0w_0) x_0  (x_1+x_0)  y_1^2  z_0^2    \cr
\end{pmatrix}.
$$\end{theorem}

\begin{proof} To check that $\varphi(\PP(\sO_{\PP^1}\oplus \sO_{\PP^1}(2))\times \PP^1 \times \PP^1 \times \PP^1) \subset V(J_{hyp})$ is an easy \Mac2  computation. To check that $\varphi$ is dominant it remains to check that the rank of the Jacobian matrix of $\varphi$ coincides with the dimension of the affine cone over $V_{hyp}$, which is $6$. To prove that the map is birational it suffices to verify that the preimage of the image of a general point in the product, coincides with the point up to parts contained in the base loci, another easy \Mac2 computation which we documented in the file 
\href{https://www.math.uni-sb.de/ag/schreyer/images/data/computeralgebra/M2/doc/Macaulay2/NumericalGodeaux/html/_verify__Thm__Hyp__Locus}{verifyThmHypLocus}.
\end{proof}

Perhaps more interesting than these computations is the way we found the parametrization. We describe our approach next.
The \Mac2 details can be obtained with our function \href{https://www.math.uni-sb.de/ag/schreyer/images/data/computeralgebra/M2/doc/Macaulay2/NumericalGodeaux/html/_compute__Parametrization__Of__Hyp__Locus}{computeParametrizationOfHypLocus}.

From the length of the resolution we see that $J_{hyp}$ is (arithmetically) Cohen-Macaulay. The dualizing sheaf of 
$V_{hyp}$ has a linear presentation matrix
$$0 \leftarrow \omega_{hyp} \leftarrow \sO_{\PP^{11}}^{20}(-2)\leftarrow \sO_{\PP^{11}}^{92}(-3)$$
which is up to a twist  the transpose $\psi$ of the last matrix in the free resolution of $J_{hyp}$. However, $\omega_{hyp}$ is not an invertible sheaf, since $V_{hyp}$ has non-Gorenstein singularities.
The first step towards the computation of the parametrization is to compute the image under the rational map 
$$V_{hyp} \dasharrow \PP^{19}$$
defined by $|\omega_{hyp}(2)|$.  The graph of this map is contained in
the scheme defined  by 
$$\begin{pmatrix} b_0 \ldots b_{19}\end{pmatrix}\psi$$
and the differentiation with respect to the $a^{(i)}_{ij}$'s gives a $12\times 92$-presentation matrix
$$0 \leftarrow \sL \leftarrow \sO_{\PP^{19}}^{12}\leftarrow \sO_{\PP^{19}}^{92}(-1).$$
The annihilator of $\sL$ is an ideal $J_1$ of degree $180$ and codimension $14$ with Betti table
$$\begin{matrix}
      &0&1\\\text{total:}&1&247\\\text{0:}&1&\text{.}\\\text{1:}&\text{.}&19\\
     \text{2:}&\text{.}&228\\\end{matrix} \qquad  \hbox{ and }  \qquad \begin{matrix}
      &0&1\\\text{total:}&1&81\\\text{0:}&1&\text{.}\\\text{1:}&\text{.}&37\\
     \text{2:}&\text{.}&44\\\end{matrix}$$
is the Betti table of the ideal $J_{1s}$ of degree $168$ and codimension $14$ obtained by saturating $J_1$ with respect to 
$\prod_{j=0}^{19} b_i$.

\begin{remark} 
On first glance, we were surprised that $J_1$ was not prime. The discovery of the other component is in principal possible via primary decomposition. We discovered them by saturating in $b_0$ by good luck. The residual part $J_{residual}=J_1:J_{1s}$
decomposes into $6$ components of degree $2$. 
The explanation of the additional components is that these are contributions from the non-Gorenstein loci of $J_{hyp}$.
\end{remark}

The next lucky discovery was that the linear strand of the resolution $J_{1s}$  has Betti table
 $$\begin{matrix}
      &0&1&2&3&4&5\\\text{total:}&1&37&84&54&24&5\\\text{0:}&1&\text{.}&\text
      {.}&\text{.}&\text{.}&\text{.}\\\text{1:}&\text{.}&37&84&54&24&5\\
 \end{matrix}.$$ 
 The cokernel of the last differential in the linear strand transposed and twisted, $\sO_{\PP^{19}}^5 \leftarrow \sO_{\PP^{19}}^{24}(-1)$,
 is supported on a rational normal scroll of degree $6$.
 Indeed, $V_1=V(J_{1s})$ is contained in a rational normal scroll which is the cone over $\PP^5\times \PP^1 \subset \PP^{11}$ with vertex a $\PP^7 \subset \PP^{19}$. Using a scrollar syzygy \cite{Bothmer}, we can compute the $2 \times 6$-matrix of linear forms defining the scroll in $\PP^{11}$, and hence an isomorphism of the scroll with  $\PP^5\times \PP^1$.
 
 The projection of $V(J_{1s})$ from the vertex into $\PP^{11}$ is defined by an ideal $J_2$   whose resolution has Betti numbers
 $$\begin{matrix}
       &0&1&2&3&4&5&6&7\\\text{total:}&1&19&55&97&99&56&20&3\\\text{0:}&1&\text{.}&\text{.}&\text{.}&\text{.}&\text{.}&\text{.}&\text{.}\\\text{1:}&\text{.}&19&52&45&24&5&\text{.}&\text{.}\\\text{2:}&\text{.}&\text{.}&3&
       52&75&36&8&\text{.}\\\text{3:}&\text{.}&\text{.}&\text{.}&\text{.}&\text{.}&15&12&3
 \end{matrix}.$$
 Using Cox coordinates $\QQ[c_0,\ldots,c_5,w_0,w_1]$ we obtain that $V_2=V(J_2)$ is a complete intersection of two quadric bundles of class $2H-R$ on $\PP^5\times \PP^1$, where $H$ and $R$ denote the hyperplane class and the ruling in
 $\Pic (\PP^5\times \PP^1)$. Following \cite{Schreyer86}, the resolution of $J_2$ can be obtained from the exact sequence
 $$ 0 \leftarrow \sO_{V(J_2)}  \leftarrow \sO_{\PP^5\times \PP^1} \leftarrow \sO_{\PP^5\times \PP^1}(-2H+R) \leftarrow \sO_{\PP^5\times \PP^1}^2(-4H+2R) \leftarrow 0$$
 via an iterated mapping cone 
 $$ [[\sC^0  \leftarrow \sC^1(-2)\oplus \sC^1(-2)] \leftarrow \sC^2(-4)] $$
 over Buchsbaum-Eisenbud complexes $\sC^i$ associated to the $2\times 6$-matrix defining the scroll
 with Betti tables
 $$\begin{matrix}
       &0&1&2&3&4&5\\
 \text{0:}&1&\text{.}&\text{.}&\text{.}&\text{.}&\text{.}\\
 \text{1:}&\text{.}&15&40&45&24&5\\
 \text{2:}&\text{.}&\text{.}&\text{.}&\text{.}&\text{.}&\text{.}\\
\text{3:}&\text{.}&\text{.}&\text{.}&\text{.}&\text{.}&\text{.}\\
 \end{matrix}
  \;+\;
 2 \cdot \begin{matrix}
       &1&2&3&4&5&6\\
 \text{0:}&\text{.}&\text{.}&\text{.}&\text{.}&\text{.}&\text{.}\\
 \text{1:}&2& 6&\text{.}&\text{.}&\text{.}&\text{.}\\
 \text{2:}&\text{.}&\text{.}&20&30&18&4\\
\text{3:}&\text{.}&\text{.}&\text{.}&\text{.}&\text{.}&\text{.}\\
 \end{matrix} 
  \;+\;\begin{matrix}&2&3&4&5&6&7\\
 \text{0:}&\text{.}&\text{.}&\text{.}&\text{.}&\text{.}&\text{.}\\
  \text{1:}&\text{.}&\text{.}&\text{.}&\text{.}&\text{.}&\text{.}\\
 \text{2:}&3&12&15&\text{.}&\text{.}&\text{.}\\
 \text{3:}&\text{.}&\text{.}&\text{.}&15&12&3\\
 \end{matrix},
$$ 
which gives the table above.

\begin{proposition}\label{paraV2} $V_2$ is birational to $V_2'=\PP(\sO_{\PP^1}\oplus \sO_{\PP^1}(2))\times \PP^1 \times \PP^1$.
\end{proposition}

\begin{proof} By inspection we find that in terms of the Cox ring $\QQ[c_0,\ldots,c_5,w_0,w_1]$ of $\PP^5\times \PP^1$ the variety $V_2$ is defined by two relative rank $4$ quadrics
$$
\det \begin{pmatrix} c_0& c_2\cr
-c_4w_1&c_5(w_0-w_1)\cr
\end{pmatrix} \hbox{ and }
\det \begin{pmatrix} c_1&c_2\cr
-c_4w_1&c_3w_0 \cr
\end{pmatrix} .
$$
Regarding
$$
\begin{pmatrix} c_0& c_2\cr
-c_4w_1&c_5(w_0-w_1)\cr
\end{pmatrix} \begin{pmatrix} y_0\cr
y_1 \cr
\end{pmatrix} =0\hbox{ and }
\begin{pmatrix}
z_0\cr 
z_1\cr
\end{pmatrix}=0
$$
as a linear system for $c_0,\ldots,c_5$ with coefficients in Cox coordinates $(w_0,w_1,y_0,y_1,z_0,z_1)$ of 
$\PP^1 \times \PP^1 \times \PP^1$ we find a $4 \times 6$-matrix $m$ with a rank $2$ kernel

$$ \sO(0,1,1)^2\oplus \sO(1,1,1)^2 \overset{m}{\longleftarrow} \sO^6 \overset{n}{\longleftarrow} \sO(0,-1,-1)\oplus \sO(-2,-1,-1).$$ 
Thus  
$$V_2'=\PP(\sO_{\PP^1}\oplus \sO_{\PP^1}(2))\times \PP^1 \times \PP^1 \dasharrow V_2 \subset \PP^5 \times \PP^1 \subset \PP^{11}$$
defined by 
$$
(v_0,v_1,w_0,w_1,y_0,y_1,z_0,z_1) \mapsto (v_0,v_1) n^t \otimes (w_0,w_1)
$$
gives  a rational parametrization of $V_2$. Note that the rational map is defined by forms of multidegree $(1,3,1,1)$, where
$\deg v_0=(1,2,0,0), \deg v_1=(1,0,0,0)$ and $ \deg w_0=\deg w_1=(0,1,0,0)$.
\end{proof}

\begin{proposition} 
$V_1$ is birational to 
 a degree $4$ rational normal curve fibration $\widetilde V_1' \to V_2'$ in a $\PP^4$-bundle $\PP(\sT)$ over
 $V_2'=\PP(\sO_{\PP^1}\oplus \sO_{\PP^1}(2))\times \PP^1 \times \PP^1$. The varieties $V_1'$ and $V_{hyp}$ are birational to $\PP(\sO_{\PP^1}\oplus \sO_{\PP^1}(2))\times \PP^1 \times \PP^1\times \PP^1$.
\end{proposition}

\begin{proof} Let $\widetilde \PP^{19}=\PP(\sO_{\PP^{11}}\oplus \sO_{\PP^{11}}^8(1)) \to \PP^{19}$ denote the blow-up of the projection center of $\PP^{19} \dasharrow \PP^{11}$. Taking  the pullback of this bundle along the rational map
$V_2'= \PP(\sO_{\PP^1}\oplus \sO_{\PP^1}(2))\times \PP^1 \times \PP^1 \dasharrow V_2  \subset \PP^{11}$
from Proposition \ref{paraV2} we obtain the bundle $\PP(\sU) \to V_2'$ with
$$ \sU=\sO_{V_2'} \oplus 8\sO_{V_2'}(1,3,1,1).$$
Let 
$$\psi': \QQ[b_0,\ldots,b_{19}] \to \QQ[u_0,\ldots,u_8,v_0,v_1,w_0,w_1,y_0,y_1,z_0,z_1] $$
be the corresponding ring homomorphism. We will solve the equations defined by $J_3=\psi'(J_{1s})$ in two steps.
Saturating  $J_3$ with respect to $u_0$ yields $12$ equations which are linear in the $u$'s. Their solution space defines a  $\PP^4$-bundle $\PP(\sT) \to V_2'$.  
\begin{center}
\begin{tikzpicture}
	\matrix(m)[matrix of math nodes,
		row sep=2.0em, column sep=2.9em,
		text height=1.5ex, text depth=0.25ex]
	       {\PP^{19}& \widetilde \PP^{19} &\widetilde \PP^{19}|_{V_2} &                        & \PP(\sU) & \PP(\sT) \\
	                     &                                &                                           &  \widetilde V_1 &               & V_1'  \\
	                     &\PP^{11}                  & V_2                                    &                         & V_2'       &         \\  
	       };
	\path[-stealth]
	(m-1-2) edge (m-1-1)
	(m-1-2) edge (m-3-2)
	(m-1-3) edge (m-3-3)
	(m-2-4) edge (m-3-3)
	(m-1-5) edge (m-3-5)
	(m-1-6) edge (m-3-5)
	(m-2-6) edge (m-3-5);
	\path[>=stealth,left hook->]
	(m-2-4) edge (m-1-3)
	(m-1-3) edge (m-1-2)
	(m-2-4) edge (m-1-3)
	(m-3-3) edge (m-3-2)
	(m-1-6) edge (m-1-5)
	(m-2-6) edge (m-1-6)
	;
	\path[-stealth]
	(m-1-5) edge[dashed] node[above]{$\psi$} (m-1-3)
	(m-3-5) edge[dashed] (m-3-3)
	(m-2-6) edge[dashed] (m-2-4);
  	\end{tikzpicture} 
\end{center}
The saturation of the image $J_4$ of $J_3$ leads to an ideal defined by $6$ relative quadrics, which we identify with  the minors of a homogeneous $2\times 4$-matrix.  We compute the $2\times 4$-matrix by using a relative scrollar syzygy \cite{Bothmer}. Luckily, the Gr\"obner basis computation in \Mac2 uses a relative scrollar syzygy as one of the basis elements of the third syzygy module. Introducing another $\PP^1$ factor with coordinates $(x_0,x_1)$ leads to a parametrization
$$ V_1'=\PP(\sO_{\PP^1}\oplus \sO_{\PP^1}(2))\times \PP^1 \times \PP^1\times \PP^1 \dasharrow \widetilde V_1 \subset \widetilde \PP^{19}.$$
Finally, substituting this parametrization into the $12\times 92$ presentation matrix of $\sL$ yields a rank $11$ matrix over the Cox ring of $V_1'$
and the syzygy of the transposed matrix yields the final parametrization of $V_{hyp}$ from Theorem \ref{hypLocus}.  
\end{proof}

\section{Torsion surfaces}
Let $X$ be a  numerical Godeaux surface with $\Tors X = \ZZ/5\ZZ$ or $\Tors X = \ZZ/3\ZZ$, and let $\tau_i, \tau_j  \in \Pic(X)$ be two different torsion elements with $\tau_i = - \tau_j$. 
As $h^0(K_X+\tau) = 1$ for any non-trivial torsion element $\tau$ in $\Pic(X)$, we can choose effective divisors $D_i \in |K_X+\tau_i|$ and  $D_j \in |K_X+\tau_j|$. 
 Then,  
 \begin{equation}\label{eq_redfib}
 C_{i,j} = D_i + D_{j} \in |2K_X|
 \end{equation}
 and $D_i$ and $D_j$ intersect in exactly one point which is a base point of $|3K_X|$.  Thus, for $T = \ZZ/3\ZZ$ we obtain a special reducible bicanonical fiber $D_1+D_2$ with $\tau_2 = - \tau_1$, whereas for $T = \ZZ/5\ZZ$ we obtain two special reducible bicanonical fibers $D_1+D_4$  and $D_2+D_3$ with $\tau_4= - \tau_1$ and $\tau_3= - \tau_2$. 
 
In the last case, Reid showed that for any $i \neq j$, $D_i$ and $D_j$ intersect in a unique point and that for three different $i,j,k$,
any two points of intersection are distinct (cf. \cite{ReidGodeaux78}, Lemma 0.1 and 0.2).
Hence, denoting by $P_{i,j}$ the intersection points of $D_i$ and $D_j$, then the two bicanonical divisors $C_{1,4}$ and $C_{2,3}$ intersect in four distinct points 
$\{P_{1,2},P_{1,3},P_{2,4},P_{3,4}\}$ which are exactly the base points of $|2K_X|$. In particular, $|2K_X|$ has no fixed components and four different base points. 

Next, we study the restriction of the tricanonical map $\phi_3\colon X \dashrightarrow \P^3$ to a divisor $C_{i,j} = D_i+D_j$. Working on the canonical model $\canmod$ instead, we may assume that any $D_i$ is an irreducible curve with $p_a(D_i)=2$. We have $|3K_X|_{D_i}| = |K_{D_i} + D_{j}|_{D_i}| = |K_{D_i} + Q|$ where $Q$ is a base point of $|3K_X|$. Thus, under the birational map $\phi_3$, the curve $D_i$ is mapped $2:1$ to a line. 


\subsection{$\ZZ/5\ZZ$-surfaces}\label{Z5surfaces}
These surfaces are completely classified:
 \begin{theorem}[see \cite{Miyaoka} or \cite{ReidGodeaux78})]
 	Numerical Godeaux surfaces with torsion group $\ZZ/5\ZZ$ form a unirational irreducible family of dimension 8. 
 \end{theorem}
Their proof shows that any $G = \ZZ/5\ZZ$ Godeaux surface arises as a quotient $Q_5/G$ of a quintic as follows: Let $\xi$ be a primitive fifth root of unity.  The group $G$ acts on $\P^3$ via
\begin{equation*}
\beta \colon (u_1:u_2:u_3:u_4) \mapsto (\xi u_1,\xi^2 u_2,\xi^3 u_3,\xi^4u_4).\\
\end{equation*}
Consider the $G$-invariant family of quintic forms 
\begin{equation}\label{eq_quinticsurface}
\begin{aligned}
f = f_{c} & =  u_1^5+u_2^5+u_3^5+u_4^5 + c_0u_1^3u_3u_4 + c_1u_1u_2^3u_3 +c_2u_2u_3^3u_4 + c_3u_1u_2u_4^3\\
& + c_4u_1^2u_2^2u_4 + c_5u_1^2u_2u_3^2 +c_6u_2^2u_3u_4^2 +c_7u_1u_3^2u_4^2
\end{aligned}
\end{equation}
with $c = (c_0,\ldots,c_7)^{t}\in \AA^{8}$. The coefficients of  $u_i^5$ are chosen to be 1 to guarantee that $Q_5= V(f)$ does not meet the four fixed points of the action. 
The quotient $Q_5/G$ is a Godeaux surface if the parameter $c \in \AA^{8}$ is chosen such that $Q_5$ has at most rational double points as singularities. 

In \cite{Stenger18}, Section 9.1, a minimal set of algebra generators for $R(X) = R(Q_5)^G$ with $Q_5$ and $G$ as above is presented. Moreover, using this set of generators, 
we  determined a choice for the  first syzygy matrix $d_1$ of $R(X)$ depending on $c \in \AA^{8}$. 
\medskip

In Theorem \ref{Z5surf} we reconstruct the family of $\ZZ/5\ZZ$ surfaces using our approach. We deduce computationally that the locus of possible lines in $Q$  for marked numerical Godeaux surfaces with $\Tors = \ZZ/5\ZZ$ is isomorphic to an $S_4$-orbit of six $\P^1\times \P^1$'s.

First, in Theorem \ref{thm_torsionfib}  and Proposition \ref{excludedCompo} we have seen that for constructing a $\ZZ/5\ZZ$-Godeaux surface we have to choose a line in $Q$ intersecting two $\P^3$'s in $V(I_3(e)) \cap Q$ in exactly one point.
So we choose two different $\P^{3}$'s in this locus and evaluate the condition that a line through two general points is completely contained in the variety $Q$ (see our procedure \href{https://www.math.uni-sb.de/ag/schreyer/images/data/computeralgebra/M2/doc/Macaulay2/NumericalGodeaux/html/_line__Conditions__Tors__Z5.html}{lineConditionsTorsZ5}).
	The resulting zero loci $W \subset \P^3 \times \P^3$ decomposes into a union of several surfaces of type $\P^{1} \times \P^1 \subset \P^{3} \times \P^3$ and $\P^0 \times \P^2 \subset \P^3 \times \P^3$ or $\P^2 \times \P^0 \subset \P^3 \times \P^3$. We verify computationally that there are two components of $W$, both isomorphic to a $\P^{1} \times \P^1$, which give lines leading generically to numerical Godeaux surfaces with torsion group $\ZZ/5\ZZ$. We display the results in a table, using the following notation: 
	\begin{align*}
n & = \text{ number of components of $W$ of a given type} ,\\
f & = \text{ dimension of the family of lines}, \\
s &= \text{ projective dimension of the linear solution space in the second step} ,\\
t  &= \text{ total dimension of the constructed family of varieties} ,\\
J &= \text{ generating ideal of the model in $\P^1 \times \P^3$ (see Remark \ref{bitrimap})}, \\
R.C. &=  \text{ the ring condition introduced in \cite[Remark 2.3]{SS20}}
\end{align*}
	\begin{figure}[h]
	\begin{tabular}{p{0.3cm}|p{1.8cm}|p{0.4cm}|l|p{2.3cm}|p{6.00cm}}
	$n$	& family of \newline lines & $f$  $s$ $\bf{t}$ & $R.C.$ & $\#$ of gen. of $J$ of a given bidegree& comments \\ \hline 
		2 & $\P^1 \times \P^1$ 	&
		 $\! \begin{aligned}	
	 & 2\\
	& 9 \\
	& \bf{11}
	 \end{aligned}$ & 
		 true &   $\! \begin{aligned}	
		\left\{0,\,7\right\} & \Rightarrow 1\\
		\left\{1,\,2\right\} & \Rightarrow 1 \\
		\left\{1,\,5\right\} & \Rightarrow 1 \\
		\left\{2,\,3\right\} & \Rightarrow 1 
		\end{aligned}$ & 
		a general $\ZZ/5\ZZ$-Godeaux surface \\
	\hline 
	4 & $\P^1 \times \P^1$ 	&
	$\! \begin{aligned}	
	& 2\\
	& 16  \\
	& \bf{18}
	\end{aligned}$ &
	false &   $\! \begin{aligned}	
	\left\{0,\,2\right\} & \Rightarrow 3
	\end{aligned}$ & 
$J$ defines a surface in $\P^1 \times \P^3$ which is the union of three  $\P^1 \times \P^1$'s \\
	\hline 
4 \newline 4
 & $\P^0 \times \P^2$ \newline  $\P^2 \times \P^0$ 	&
	$\! \begin{aligned}	
& 2\\
& 16 \\ 
& \bf{18} & \end{aligned}$ & 
false &   $\! \begin{aligned}	
\left\{0,\,2\right\} & \Rightarrow 3
\end{aligned}$ & 
$J$ defines a surface in $\P^1 \times \P^3$ which is the union of three  $\P^1 \times \P^1$'s
	\end{tabular}
\caption{$\ZZ/5\ZZ$-surfaces}
\label{tab:Z5surfaces}
	\end{figure}
\begin{remark}\label{bitrimap}
Our Macaulay2-package \cite{SS20M2} contains procedures for various birational models of the computed surface.  Recall that the tricanonical map defines a birational map $\phi_3$ to $\P^3$. We can compute the image of the projection map $\canmod \rightarrow \Proj(S) = \P(2^2,3^4)$, the image of $\phi_3$ in $\P^3$ and the image of the product of the bi- and tricanonical map to $\P^1 \times \P^3$ which is also birational onto its image:
\begin{center}
	\begin{tikzpicture}
	\matrix (m) [matrix of math nodes, row sep=3.6em,
	column sep=3.5em,text height=1.5ex, text depth=0.25ex]
	{  
		\tilde{X}& X    &  \canmod & \P(2^2,3^4,4^4,5^3) \\
		\P^1& \P^1 \times \P^3 &  \P(2^2,3^4) & \\};
	\path[-stealth]
	(m-1-1) edge (m-1-2)
	edge node[left] {} (m-2-1)
	edge (m-2-2)
	(m-1-3) edge node[left] {} (m-2-3)
	(m-2-2) edge [dashed] (m-2-1)
	(m-2-3) edge[dashed] (m-2-2)
	(m-1-2) edge node[above] {$\pi$} (m-1-3)
	edge[dashed] node[left] {$g$}(m-2-2)       
	edge (m-2-3);
	\path[right hook->] (m-1-3) edge (m-1-4);         
	\end{tikzpicture}
\end{center}
As before, $\tilde{X}$ denotes the blow-up of the 4 bicanonical base points of a marked numerical Godeaux surface.
	Note that for $\Tors X = \ZZ/5\ZZ$, the image under the tricanonical map is a hypersurface of degree 7 in $\P^3$. Hence the ideal of the image of $X$ under the map $g$ to $\P^1 \times \P^3$ must contain a form of bidegree $(0,7)$. 
	\end{remark}
\begin{example}
	For example, lines in $Q \subset \PP^{11}$ given in Stiefel coordinates
	\setcounter{MaxMatrixCols}{20}
	\[
	\begin{pmatrix}
	0& 0& 0& p_0& 0& 0& 0& 0& p_1& 0& 0& 0 \\ 	
	0& q_0& 0& 0& 0& 0& 0& 0& 0& 0& q_1& 0
	\end{pmatrix},
	\]
	where $((p_0,p_1),(q_0,q_1)) \in \P^1 \times \P^1$, lead to $\ZZ/5\ZZ$- Godeaux surfaces whose two reducible fibers $C_{1,4}$ and $C_{2,3}$ are mapped to the union of lines $V(y_0,y_1) \cup V(y_2,y_3)$ and $V(y_0,y_2)\cup V(y_1,y_3)$ under the tricanonical map.
\end{example}
 
\begin{theorem}\label{Z5surf}
	There is subscheme of the Fano variety $F_1(Q)$ of lines in $Q$ isomorphic to an $S_4$-orbit of $6 = 2\binom{3}{2}$  surfaces of type $\P^1 \times \P^1$ whose elements lead generically to Godeaux surfaces with  $\Tors= \ZZ/5\ZZ$. Moreover, each surface gives a $2+9$-dimensional unirational family in the unfolding parameter space for marked Godeaux surfaces, which modulo the $(\CC^{*})^{3}$-action gives an $8$-dimensional family of Godeaux surfaces with $\Tors= \ZZ/5\ZZ$.
\end{theorem}
\begin{proof}
	The first part of the statement follows from the explanations from the beginning of the section. Indeed, we have a choice of $\binom{3}{2}$ $\PP^3$s which our line $\ell$ in $Q$ has to intersect. Then, for each choice we have 2 surfaces  isomorphic to $\PP^1 \times \PP^1$ which parametrize the corresponding lines. We verified computationally that these surfaces form one $S_4$-orbit of the group action described in \cite[Proposition 4.4]{SS20}. Now, choosing such a line, we compute in an example over $\mathbb{Q}$ that the solution space in the second step is 9-dimensional. Note that this is the minimal dimension which can occur as the number of moduli of a numerical Godeaux surface is 8 and we then obtain a family of dimension $8=2 + 9 -3 $  over an open subset where the minimum is attained. 
\end{proof}

\subsection{$\ZZ/3\ZZ$-surfaces}\label{Z3surfaces}
As in the case $\Tors = \ZZ/5\ZZ$,  surfaces with $\Tors = \ZZ/3\ZZ$ are completely classified:
\begin{theorem}[see \cite{ReidGodeaux78}]
Numerical Godeaux surfaces with $\Tors = \ZZ/3\ZZ$ form an irreducible, unirational $8$-dimensional component of the moduli space of Godeaux surfaces. 
\end{theorem}
As in the previous section, the idea is to start with a covering $Y \rightarrow X$, where $Y$ is a surface of general type with $K^2 = 3, p_g =2$ and $q=0$ on which the group $G= \ZZ/3\ZZ$ acts freely. Reid completely describes those covering surfaces in \cite{ReidGodeaux78} and gives a refined description for the equations of $Y$ using unprojection in \cite{Reid13}.

In \cite{Stenger18}, using the unprojection method from \cite{Reid13}, we construct a general cover surface $Y$ and the canonical ring $R(X) = R(Y)^G$ and verified that there are four distinct base points of $|2K_X|$ and that the canonical model $\canmod$ is smooth at these base points. Hence, we may assume that a general  $\ZZ/3\ZZ$-Godeaux surface admits a marking.

As in the previous section, we want to parametrize lines in $Q$ leading to $\ZZ/3\ZZ$-Godeaux surfaces. Recall that the tricanonical system of a surface $X$ with $\Tors = \ZZ/3\ZZ$ has a single base point contained in a unique bicanonical fiber $C_{1,2} = D_1+D_2$. 
Thus, by Theorem \ref{thm_torsionfib}  and Proposition \ref{excludedCompo}, we have to construct lines in $Q$ meeting a unique $\P^3 \subset V(I_3(e)) \cap Q$. We first choose a general point $p$ in such a $\P^3$ and then a second point in the cone
\[ Z = T_pQ \cap Q \subset \P^{11}. \]
In each irreducible component $Z_i$ of $Z$ we choose a point $q$ and examine the surface constructed from the line $\ell = \overline{pq}$. We present the computational results in a table, using
\[Z_i = \text{ irreducible component of $Z = T_p(Q) \cap Q \subset \P^{11}$}\]
and $f,s,t$ and $J$ are defined as in Table 
\eqref{tab:Z5surfaces}.
\smallskip

\begin{tabular}{l|p{1.35cm}|p{0.3cm}|l|p{2.3cm}|p{6.19cm}}
	& $\betti Z_i$ 
	$\#$ of $Z_i$  & $f$  $s$  $\textbf{t}$ & $R.C.$ & $\#$ of gen. of $J$ of a given bidegree& comments \\ \hline 
	1.1) &  $\begin{matrix}
	&0&1\\
	\text{0:}&1&6\\\text{1:}&\text{.}&3\\\end{matrix}$ \newline  \newline  $1$
	& 5  6 \textbf{11} & true & 
	$\! \begin{aligned}	
	\left\{0,\,8\right\} & \Rightarrow 1\\
	\left\{1,\,3\right\} & \Rightarrow 1 \\
	\left\{1,\,5\right\} & \Rightarrow 1 \\
	\left\{2,\,2\right\} & \Rightarrow 1 
	\end{aligned}$ & 
	a general $\ZZ/3\ZZ$ Godeaux surface
	\\ \hline
		1.2) & 
	$\begin{matrix}
	&0&1\\
	\text{0:}&1&8\\\end{matrix}$ \newline \newline $4$
	& 5  6  \textbf{11} & true & 
	$\! \begin{aligned}	
	\left\{0,\,2\right\} &\Rightarrow 1\\
	\left\{3 ,\,4\right\} &\Rightarrow 2 \\
	\left\{4 ,\,1\right\} &\Rightarrow 1
	\end{aligned}$ & $J$ defines a surface in $\P^1 \times \P^3$ which decomposes into a union of two $(3,3)$-hypersurfaces in $\P^1 \times \P^2$  \\ \hline 
	1.3)&
	 $\begin{matrix}
	&0&1\\
	\text{0:}&1&8\end{matrix}$ \newline \newline $4$
	& 5  12 \textbf{17} & false & 
	$\! \begin{aligned}	
	\left\{0,\,2\right\} & \Rightarrow 2\\
	\left\{1,\,2\right\} & \Rightarrow 1 \\
	\end{aligned}$ & 
	$J$ defines a surface in $\P^1 \times \P^3$ which decomposes into a $\P^1 \times \P^1$ and a $(1,2)$-hypersurface in $\P^1 \times \P^2$.  \\ \hline 
	1.4)&
  $\begin{matrix}
	&0&1\\
	\text{0:}&1&8\end{matrix}$ \newline \newline $1$
	& 4  31 \textbf{35} & false & 
	$\! \begin{aligned}	
	\left\{0,\,2\right\} & \Rightarrow 4 \\
	\end{aligned}$ & 
	$J$ defines a union of two $\P^1 \times \P^1$'s  in $\P^1 \times \P^3$ \\
\end{tabular}

\begin{remark}
Only the first example leads to a numerical Godeaux surface $X$ with $\Tors X = \ZZ/3\ZZ$.  Note that the image of $X$ under the tricanonical map is a hypersurface of degree 8 in $\P^3$. Hence the ideal of the image of $X$ under the product of the bi- and tricanonical map to $\P^1 \times \P^3$ must contain a form of bidegree $(0,8)$. 
\end{remark}
Analyzing the component $Z_1 \subset T_p(Q)$ from the case 1.1) we see that it is cone with vertex $p$ over a singular surface in $\P^4$ with minimal free resolution 
\[
\begin{matrix}
&0&1&2\\\text{total:}&1&3&2\\\text{0:}&1&\text{.}&\text{.}\\\text{1:}&\text{.}&3&2\\\end{matrix}.
\]
Using this observation, we are able to rationally parametrize lines for the case 1.1.):
\begin{theorem}
	There exists an $S_4$-orbit of three unirational 5-dimensional  subschemes  of $F_1(Q)$ whose elements lead generically to  numerical Godeaux surfaces with $\Tors X = \ZZ/3\ZZ$. One of them is parameterized by 
\[ 
 \varphi \colon \P(\sO_{\P^3 \times \P^1}(2,3)  \oplus \sO_{\P^3 \times \P^1}) \dashrightarrow F_1(Q)  \]
 \[
(u,w,z)  \mapsto 
	\langle \begin{pmatrix}
		{u}_{0}\\
		0\\
		0\\
		{u}_{1}\\
		0\\
		0\\
		0\\
		0\\
		{u}_{2}\\
		0\\
		0\\
		{u}_{3}\end{pmatrix}, \begin{pmatrix}
		{{u}_{0}^{2}{u}_{1}{w}_{1}^{3}{z}_{1}}\\
		{{u}_{1}{u}_{3}^{2}{w}_{0}^{2}{w}_{1}{z}_{1}}\\
		-{u}_{1}{u}_{2}^{2}{w}_{0}^{2}{w}_{1}{z}_{1}\\
		0\\
		{{u}_{0}{u}_{3}^{2}{w}_{0}^{2}{w}_{1}{z}_{1}}\\
		-{u}_{0}{u}_{2}^{2}{w}_{0}^{2}{w}_{1}{z}_{1}\\
		-{u}_{1}^{2}{u}_{3}{w}_{0}{w}_{1}^{2}{z}_{1}\\
		{{u}_{0}^{2}{u}_{3}{w}_{0}{w}_{1}^{2}{z}_{1}}\\
		{u}_{2}{z}_{0}-{u}_{2}^{2}{u}_{3}{w}_{0}^{3}{z}_{1}\\
		-{u}_{1}^{2}{u}_{2}{w}_{0}{w}_{1}^{2}{z}_{1}\\
		{{u}_{0}^{2}{u}_{2}{w}_{0}{w}_{1}^{2}{z}_{1}}\\
		{{u}_{3}{z}_{0}}\end{pmatrix} \rangle.
	\]
	Moreover, there exists a $5+6$-dimensional unirational family in the unfolding parameter space for marked Godeaux surfaces, which modulo the $(\CC^{*})^{3}$-action gives a $8$-dimensional family of Godeaux surfaces with Torsion group $T= \ZZ/3\ZZ$.
\end{theorem}
\begin{proof}
First, we choose one of the three components in $V(I_3(e)) \cap Q$ isomorphic to a $\P^3$, for example
	\[
V = V(a_{3,1}, a_{3,0}, a_{2,1},a_{2,0},
	a_{1,3} , a_{1,2},a_{0,3}, a_{0,2} ) \subset Q.
	\]
	Let $u_0,\ldots,u_3$ be new homogeneous coordinates, then a general point $p$ in $V$ is given by
	\[
p = {\left({\begin{array}{cccccccccccc}
		{u}_{0}&0&0&{u}_{1}&0&0&0&0&{u}_{2}&0&0&{u}_{3}\\
		\end{array}}\right)}^t
	\]
and the tangent space is isomorphic to a $\P^7$. Computing the equations,  we obtain a $\P^7$-bundle
	\[ \P(\sO_{\PP^{3}}^4 \oplus \sO_{\PP^{3}}(1)^4) \]
	with homogeneous coordinates $v_0,\ldots,v_7$. 
	Next, we consider the restriction of $Q$ to this $\P^7$-bundle  and choose the irreducible component $\tilde{Z_1}$ of $Q$ specializing to the component $Z_1$ described above.  
	Among the defining equations $\tilde{Z_1}$, there are two determinantal equations given by
	\[
	\begin{pmatrix}
	 u_0^2 &  v_7 \\     
	 u_1^2  & -v_6          
	\end{pmatrix} \text{ and }
		\begin{pmatrix} 
	u_2^2 &  v_5 \\
	u_3^2  &-v_4 
	\end{pmatrix}.
	\]
	Using the substitution 
  \begin{equation*}
  v_7  = v_{6,7}u_0^2 , \ \  v_6   = - v_{6,7}u_1^2,  \ \
   v_5  = v_{4,5}u_2^2,\ \  v_6   = - v_{4,5}u_3^2 
  \end{equation*}
	with new coordinates in $v_{6,7}$ and $v_{4,5}$, we obtain a new variety $Z_1'$ in a $\P^1$-bundle over $\P(\sO_{\PP^{3}}^4)$ defined by 3 equations whose minimal free resolution is given by the Hilbert-Burch matrix 
	 \[m = \begin{pmatrix}
	-{u}_{1}{v}_{0}+{u}_{0}{v}_{1}&{-{u}_{0}^{2}{u}_{1}^{2}{v}_{67}}\\
	{u}_{2}^{2}{u}_{3}^{2}{v}_{45}&{u}_{3}{v}_{2}-{u}_{2}{v}_{3}\\
	{v}_{67}&{v}_{45}\end{pmatrix}.\]
Now regarding 
\[ m \begin{pmatrix}
w_0\cr 
w_1\cr
\end{pmatrix}=0  \]
as a linear system for $v_0,\ldots,v_3,v_{4,5},v_{6,7},$ with coefficients in the coordinates $(u,w_0,w_1)$ of $\P^3 \times \P^1$ we  
obtain a $3 \times 6$ matrix $m$ with a rank $3$ kernel
\[ \sO(2,1)^2\oplus \sO(-2,1) \overset{m}{\longleftarrow} \sO(1,0)^4 \oplus \sO(-2,0)^2 \overset{n}{\longleftarrow} \sO^2\oplus \sO(-2,-3) \] 	
and
\[ n^t = 
\begin{pmatrix}
u_{0}&u_{1}&0&0&0&0\\
0&0&u_{2}&u_{3}&0&0\\
u_{0}^{2}u_{1}w_{1}^{3}&0&-u_{2}^{2}u_{3}w_{0}^{3}&0&-w_{0}w_{1}^{2}&w_{0}^{2}w_{1}\\
\end{pmatrix}
\]
As the sum of the first and second line gives the point $p$ (in the corresponding coordinates) and we are looking for a second point $q \in T_p{Q}$ with $q \neq p$, we obtain only 2-dimensional solution spaces for $v_0,\ldots,v_3,v_{4,5},v_{6,7},$. We choose the space spanned by the last two columns of $n$ which has the right dimension for general points $(u,w) \in \P^3 \times \P^1$. Thus, we obtain a birational map
\begin{align*}
\P(\sO_{\P^3 \times \P^1}(2,3)  \oplus \sO_{\P^3 \times \P^1})  & \dashrightarrow Z_1' \\
(u_0,\ldots,u_3,w_0,w_1,z_0,z_1) & \mapsto \begin{pmatrix}
0 &z_0 &z_1
\end{pmatrix}n^t.
\end{align*}
Reversing the single substitution steps, we obtain the desired parametrization for the second point $q$ and hence for a line $\ell = \overline{pq} \subset Q$. Moreover, we verify that for a general point $(u,w,z)$ the line specializes to one chosen in the case 1.1) above which leads to a numerical Godeaux surface with $\Tors = \ZZ/3\ZZ$.  The rest of the proof and the dimension count of the constructed family is now exactly as in Theorem \ref{Z5surf}.
\end{proof}

\section{Special surfaces and ghost components}\label{ghost}
Recall that any line in $Q$ meeting none of the special loci described in Section \ref{sec_prelim} leads to a member of our constructed dominant family of \cite{SS20}, that means to a torsion-free marked numerical Godeaux surface with no hyperelliptic fiber. Marked numerical Godeaux surfaces with a non-trivial torsion group form an irreducible family and their associated lines in $Q$ were described in Section \ref{Z5surfaces} and Section \ref{Z3surfaces}.  Thus, the question remains whether there exists a different family of torsion-free marked numerical Godeaux surfaces. The associated lines of such a family must necessarily meet one of the special loci. 

\subsection{Lines meeting one special locus}
In the following, we represent our computational experiments with \Mac2 for lines in $Q$ meeting one of the special loci. We work over the finite field $\FF_{32233}$ and proceed as follows. First we choose for each type of the the special loci a representative (under the $S_4$-action) and compute a general point $p$ in this component. Afterwards, we determine the intersection 
\[ Z = T_p(Q) \cap Q \subset T_p(Q) \subset \P^{11},\]
where $Z$ is a cone with vertex $p$.  Except for points in the hyperelliptic locus, we can decompose $Z$ into its irreducible components. For general points in the hyperelliptic locus we believe that $Z$ is an irreducible variety. Next we choose a general point $q$ in each of the irreducible components of $Z$ and obtain a line $\ell = \overline{pq} \subset Q$. Finally, we compute the solution space and determine the dimension of the constructed family of varieties. 

Let $F$ be the obtained complex with Betti numbers as in Equation \eqref{bettinumbers}, and let $d_1$ be the first syzygy matrix.
 In the end, we verify whether the resulting module $R = \coker d_1$ satisfies the ring condition $(R.C.)$ from \cite[Remark 2.3]{SS20} and compute the bihomogeneous model of the variety $V( \ann \coker d_1)$ in $\P^1 \times \P^3$. 

Recall that if $R$ is the canonical ring of a numerical Godeaux surface $X$, then the ideal of the model in  $\P^1 \times \P^3$ must contain a form of bidegree
$(0,9 - \# \text{ of base points of } 3K_X)$, hence of bidegree $(0,9)$, $(0,8)$ or $(0,7)$. 

We represent the results in a table 
denoting by $W$ the chosen special locus in $Q$ and using the notation $f,s,t,J$ and $Z_i$ as in Table \eqref{tab:Z5surfaces}.
Note that if $Z$ has several irreducible components leading to families with the same values for $f,s$ and $t$ and the same type of irreducible components of the model in $\P^1 \times \P^3$, we state the number of those components and display only the computational results for one such component.
\begin{figure}
\begin{tabular}{l|p{1.6cm}|p{1.8cm}|p{1.6cm}|p{0.33cm}|l|p{1.8cm}|p{3.27cm}}
	&$\betti W$ $\dim W$ &occurrence of $W$ & $\betti Z_i$ 
	$\#$ of $Z_i$ & $f$  $s$  $\textbf{t}$ & $R.C.$ & $\#$ of gen. of $J$ of a given bidegree& comments \\ \hline 
	1) & 
	$\begin{matrix} &0&1\\
	\text{0:}&1&6\\
	\text{1:}&\text{.}&1\\\end{matrix}$  \newline \newline  $4$
 &
		 support of $H_0(C_1)$, 
		 support of $V(m_3(e))$
	& 
	 $\begin{matrix}
		&0&1\\
		\text{0:}&1&6\\
		\text{1:}&\text{.}&1\\\end{matrix}$
		\newline \newline $3$
	 & 7  7   \textbf{14} & false &
   $\! \begin{aligned}	
   \left\{0,\,2\right\} &\Rightarrow 2\\
  \left\{4,\,4\right\} &\Rightarrow 1
\end{aligned}$ & 
$J$ defines a surface in $\P^1 \times \P^3$ which decomposes into a  $\P^1 \times \P^1$ and a $(4,4)$-hypersurface in $\P^1 \times \P^2$.
\\ \hline
 2.1) &	$\begin{matrix} &0&1\\
\text{0:}&1&5\\
\text{1:}&\text{.}&1\\\end{matrix}$  \newline \newline $5$
&
support of $H_1(C_2)$ & 
$\begin{matrix}
&0&1\\
\text{0:}&1&4\\\text{1:}&\text{.}&7\\\end{matrix}$ \newline \newline $1$
& 7   3  \textbf{10} & true &  
$\! \begin{aligned}	
\left\{0,\,9\right\} &\Rightarrow 1\\
\left\{1,\,6\right\} & \Rightarrow 3 \\
& \vdots  \\
\left\{6,\,2\right\} & \Rightarrow 3 \\
\end{aligned}$ & contained in the closure of the dominant component
\\ \hline
2.2) &
& & 
$\begin{matrix}
&0&1\\
\text{0:}&1&6\\\text{1:}&\text{.}&1\\\end{matrix} \newline \newline $1$
$
&7  7  \textbf{14} & false & 
$\! \begin{aligned}	
\left\{0,\,2\right\} &\Rightarrow 2\\
\left\{4,\,4\right\} &\Rightarrow 1
\end{aligned}$ &  	
as in 1) \\ \hline
	3.1) &	$\begin{matrix} &0&1\\
	\text{0:}&1&4\\
	\text{1:}&\text{.}&2\\\end{matrix}$  \newline \newline $5$
	&
	support of $H_1(C_2)$ & 
	$\begin{matrix}
	&0&1\\
	\text{0:}&1&4\\\text{1:}&\text{.}&5\\\end{matrix}$ \newline \newline $1$
	& 7  3  \textbf{10} & true &  
	$\! \begin{aligned}	
	\left\{0,\,9\right\} &\Rightarrow 1\\
	\left\{1,\,6\right\} & \Rightarrow 2 \\
	& \vdots \\
	\left\{7,\,2\right\} & \Rightarrow 1 \\
	\end{aligned}$ & a torsion-free numerical Godeaux surface with no hyp. fibers
	\\ \hline 
	3.2) &
	&  & 
	$\begin{matrix}
	&0&1\\
	\text{0:}&1&8\\\end{matrix}$ \newline \newline $1$
	& 6  5  \textbf{11} & true & 
	$\! \begin{aligned}	
	\left\{0,\,2\right\} &\Rightarrow 1\\
	\left\{3,\,4\right\} &\Rightarrow 2\\
	\left\{5,\,3\right\} &\Rightarrow 1
	\end{aligned}$ &  	
	surface in $\P^1 \times \P^3$ decomposes into a union of two surfaces which are $(3,3)$ hypersurfaces in $\P^1 \times \P^2$ 
	\\ \hline 
	3.3) &
	&  & 
	$\begin{matrix}
	&0&1\\
	\text{0:}&1&8\\\end{matrix}$ \newline \newline $2$
	& 6  4  \textbf{10} & true & 
	$\! \begin{aligned}	
	\left\{0,\,2\right\} &\Rightarrow 1\\
	\left\{3,\,4\right\} &\Rightarrow 2\\
	\left\{5,\,3\right\} &\Rightarrow 1
	\end{aligned}$ &  	
	as in 2.2)
	\\ \hline 
	3.4) &
	&  & 
	$\begin{matrix}
	&0&1\\
	\text{0:}&1&8\\\end{matrix}$ \newline \newline $2$
	& 6  3  \textbf{9} & true & 
	$\! \begin{aligned}	
	\left\{0,\,2\right\} &\Rightarrow 1\\
	\left\{3,\,4\right\} &\Rightarrow 2\\
	\left\{5,\,3\right\} &\Rightarrow 1
	\end{aligned}$ &  	
	as in 2.2) 
\end{tabular}
\end{figure}
\newpage 
\noindent
\begin{figure}[h]
\begin{tabular}{l|p{1.8cm}|p{1.8cm}|p{1.8cm}|p{0.3cm}|l|p{1.8cm}|p{3.09cm}}
 4) &	$\begin{matrix} &0&1\\
 \text{0:}&1&3\\
 \text{1:}&\text{.}&4\\
 \text{2:}&\text{.}&1\\\end{matrix}$  \newline \newline $4$
 &
 support of $H_0(C_1)$, contained in the hyperelliptic locus & 
 $\begin{matrix}
 &0&1\\
 \text{0:}&1&4\\
 \text{1:}&\text{.}&4\\\end{matrix} \newline \newline 1
 $
 & 6  3  \textbf{9} & true &  
 $\! \begin{aligned}	
 \left\{0,\,9\right\} &\Rightarrow 1\\
 \left\{1,\,6\right\} & \Rightarrow 5 \\
 & \vdots \\
 \left\{5,\,2\right\} & \Rightarrow 1\\
 \end{aligned}$ & torsion-free numerical Godeaux surface with one hyperelliptic fiber  \\ \hline 
 5) &	$\begin{matrix} &0&1\\
 \text{0:}&1&\text{.}\\
 \text{1:}&\text{.}&4\\
 \text{2:}&\text{.}&4\\
 \text{5:}&\text{.}&4\\\end{matrix}$  \newline \newline $5$
&
 support of $H_0(C_2)$, the hyperelliptic locus & 
 $\begin{matrix}
 &0&1\\
 \text{0:}&1&4\\
 \text{1:}&\text{.}&4\\\end{matrix} \newline \newline $1$
 $
 & 7 3  \textbf{10} & true &  
 $\! \begin{aligned}	
 \left\{0,\,9\right\} &\Rightarrow 1\\
 \left\{1,\,6\right\} & \Rightarrow 5 \\
 & \vdots \\
 \left\{5,\,2\right\} & \Rightarrow 1\\
 \end{aligned}$ & torsion-free numerical Godeaux surface with one hyperelliptic fiber  \\
\end{tabular}
\end{figure}
\begin{remark}
	Note that lines meeting the components in the vanishing locus of the $3\times 3$ minors of $e$ with Betti table 
	\[	\begin{matrix}	&0&1\\\text{total:}&1&8\\\text{0:}&1&8
		& & \end{matrix}\] 
have already been studied in Section \ref{Z3surfaces}.
\end{remark}
\subsection{Lines meeting two special loci}
In this section, we present our computational experiments for lines meeting at least two special loci.  We work over the finite field $\FF_{32233}$. Lines meeting components of the vanishing locus of the $3\times 3$ minors of the $e$-matrix have already been studied in the previous sections. They either lead to marked numerical Godeaux with a non-trivial torsion group or to surfaces which are no numerical Godeaux surfaces (Theorem \ref{thm_torsionfib} and Proposition \ref{excludedCompo}).
 Thus, we exclude lines meeting any of these components from our study. 
 
 Note that  meeting two special loci in $Q$ induces at least a codimension 1 condition. Thus, being interested in further components or ghost components of our construction space, we will only display results, where the projective dimension of the linear solution space in the second step is at least 4.  Let $W_1$ and $W_2$ denote the two chosen special loci in $Q$. Then we choose first a general point $p \in W_1 \subset Q$. Afterwards we compute (if possible) the irreducible components $Z_i$ of 
\[   W_2 \cap T_p(Q) \subset Q \subset \PP^{11} \] and choose in each component $Z_i$ a general point $q$. The line $\ell = \overline{pq}$ intersects the two chosen loci in at least one point each. Note that in several cases the resulting line $\ell$ intersects also other loci non-trivially or may also be contained completely in some of the loci.  Thus, even for the same choice of pair $(W_1,W_2)$ we can obtain linear solution spaces of different dimensions depending on the choice of the component $Z_i$. 
\par\bigskip
\noindent
\begin{tabular}{l|p{1.43cm}|p{1.43cm} |p{0.25cm}|l|p{1.8cm}|p{5.05cm}} 
		&$\betti W_1$  &  $\betti W_2$ 
		& 
	  $s$  & $R.C.$ & $\#$ of gen. of $J$ of a given bidegree  
	  & comments \\ \hline 
	  1.) & 
	  $\begin{matrix}
	  &0&1\\
	  \text{0:}&1&3\\
	  \text{1:}&\text{.}&4\\
	  \text{2:}&\text{.}&1\end{matrix}
	  $ & $\begin{matrix} &0&1\\
	  \text{0:}&1&4\\
	  \text{1:}&\text{.}&2\\\end{matrix}$ & 4 & false & $\! \begin{aligned}	
	  \left\{0,\,2\right\} &\Rightarrow 1\\
	  \left\{1,\,2\right\} &\Rightarrow 1
	  \end{aligned}$ & $J$ defines a surface in $\P^1 \times \P^3$ which decomposes into a  $\P^1 \times \P^1$ and a non-complete intersection surface  in $\P^1 \times \P^3$. \\ \hline
	  2.1) & 
	  $\begin{matrix} &0&1\\
	  \text{0:}&1&4\\
	  \text{1:}&\text{.}&2\\\end{matrix}$ 
	   & 
	  $\begin{matrix} &0&1\\
	  \text{0:}&1&4\\
	  \text{1:}&\text{.}&2\\\end{matrix}$ 
	  & 4  \newline or 5 \newline & true & 
	$\! \begin{aligned}	
	\left\{0,\,2\right\} &\Rightarrow 1\\
	\left\{3,\,4\right\} &\Rightarrow 2\\
	\left\{5,\,3\right\} &\Rightarrow 1
	\end{aligned}$  &
	  $J$ defines a surface in $\P^1 \times \P^3$ which decomposes into a union of two  $(3,3)$-hypersurfaces in $\P^1 \times \P^2$. 
	  \\ \hline
	    2.2) & 
	  $\begin{matrix} &0&1\\
	  \text{0:}&1&4\\
	  \text{1:}&\text{.}&2\\\end{matrix}$ 
	  & 
	  $\begin{matrix} &0&1\\
	  \text{0:}&1&5\\
	  \text{1:}&\text{.}&1\\\end{matrix}$ 
	 & 4  \newline or 5  & true & 
	  as in 2.1) & 
	  as in 2.1)
	  \\ \hline
	  2.3) & 
	  $\begin{matrix} &0&1\\
	  \text{0:}&1&4\\
	  \text{1:}&\text{.}&2\\\end{matrix}$ 
	  & 
	  $\begin{matrix} &0&1\\
	  \text{0:}&1&4\\
	  \text{1:}&\text{.}&2\\\end{matrix}$ 
	& 4  \newline or 5 & true & 
	 $\! \begin{aligned}	
	 \left\{0,\,2\right\} &\Rightarrow 1\\
	 \left\{3,\,4\right\} &\Rightarrow 2\\
	 \left\{4,\,3\right\} &\Rightarrow 1
	 \end{aligned}$  &
	 as in 2.1)
	\end{tabular}

Depending on the dimension of the family of lines (and the group operation of $(\CC^*)^3$), the cases presented above may lead to new ghost components.

For the sake of completeness, we present the results for lines meeting the hyperelliptic locus in two different points. Recall that the hyperelliptic locus $V_{hyp}$ is a 5-dimensional subscheme of $Q$. For a point $p \in V_{hyp}$, the intersection $V_{hyp} \cap T_p(Q)$ is a curve through $p$ of degree 72 in $T_p(Q) \cong \PP^7$ (a general point in the hyperelliptic locus is a smooth point of $Q$). Hence, we have 1-dimensional choice for the second point $q$, and thus in total a 6-dimensional family of lines. Furthermore, we verified that  $(\CC^*)^3$ operates with a trivial stabilizer on the general line in this family. We obtain a 6-dimensional family of torsion-free numerical Godeaux surfaces with two hyperelliptic fibers with the following numerical data:
\par \medskip
\noindent
\begin{tabular}{l|p{1.43cm}|p{1.43cm} |p{0.25cm}|l|p{1.8cm}|p{5.05cm}} 
	3.) & 
	$\begin{matrix} &0&1\\
	\text{0:}&1&\text{.}\\
	\text{1:}&\text{.}&4\\
	\text{2:}&\text{.}&4\\
	\text{5:}&\text{.}&4\\\end{matrix}$ 
	 & 
	$\begin{matrix} &0&1\\
	\text{0:}&1&\text{.}\\
	\text{1:}&\text{.}&4\\
	\text{2:}&\text{.}&4\\
	\text{5:}&\text{.}&4\\\end{matrix}$ 
	& 3 & true & 
	$\! \begin{aligned}	
	\left\{0,\,9\right\} &\Rightarrow 1\\
	\left\{1,\,4\right\} & \Rightarrow 1 \\
	& \vdots \\
	\left\{3,\,2\right\} & \Rightarrow 1 \\
	\end{aligned}$ 
	& a numerical Godeaux surface with trivial torsion group and two hyperelliptic fibers	
\end{tabular}
\par \medskip
\begin{remark}
	It is known that the Barlow surfaces have two hyperelliptic bicanonical
	fibers.  Thus the 2-dimensional locus of Barlow lines is a sublocus of our 6-dimensional family.
\end{remark}
\begin{remark}
	Note that in the examples of torsion-free marked numerical Godeaux surfaces presented in this section, the ideal $J$ of the bi-tri-canonical model in $\PP^1 \times \PP^3$ contains a form of bidegree $(7-2h,2)$, where $h$ is the number of hyperelliptic bicanonical fibers. Using a different approach, this computational observation is proven in \cite[Theorem 2.5]{CatanesePignatelli}. 
\end{remark}

\section*{Summary}
Up to now, the question  whether  there are further marked numerical Godeaux surfaces whose bicanonical system on  the 
canonical model has 4 distinct base points stays open. If such surfaces exist, then they have trivial torsion group. The corresponding lines $\ell \subset Q \subset \PP^{11}$ have to intersect some of the special loci of Section \ref{sec_prelim}.  Experimentally over finite fields we investigated surfaces which arise from a general point on one or two of these loci. Several ghost components of such surfaces were found. However, all of the surfaces which we discovered, and which do not lie in the closure of the dominant component, are reducible and mapped to a quadric in $\PP^{3}$. If this is always true then there are no further numerical Godeaux surfaces whose bicanonical system has no fixed part and four distinct base points.

 Whether our approach will lead eventually to a complete classification of (marked) numerical Godeaux surfaces depends on  whether we will be able to exclude the existence of (even more) special lines, which lead to smooth surfaces.
In those cases the number of choices in the Construction second step has to be even larger than in the corresponding case observed so far experimentally in Section \ref{ghost}.

\bigskip

\vbox{\noindent Author Addresses:\par
\smallskip
\noindent{Frank-Olaf Schreyer}\par
\noindent{Mathematik und Informatik, Universit\"at des Saarlandes, Campus E2 4, 
D-66123 Saarbr\"ucken, Germany.}\par
\noindent{schreyer@math.uni-sb.de}\par
\smallskip

\noindent{Isabel Stenger}\par
\noindent{Mathematik und Informatik, Universit\"at des Saarlandes, Campus E2 4, 
D-66123 Saarbr\"ucken, Germany.}\par
\noindent{stenger@math.uni-sb.de}\par

}

\end{document}